\documentclass{elsarticle}
\usepackage{tikz}
\usepackage{url}
\usepackage{xcolor}
\usepackage{enumerate}
\usepackage{amssymb}
\usepackage{amsthm}
\usepackage{amsmath}

\renewcommand{\supset}{\supseteq}

\newcommand{\qvert}[1]{\draw (#1) circle (5pt)}
\tikzset{xedge/.style={line width=2.5pt}}
\tikzset{medge/.style={line width=2.5pt}}
\tikzset{triangle1/.style={fill=lightgray}}
\pgfdeclarelayer{edges}
\pgfdeclarelayer{triangles}
\pgfdeclarelayer{toplayer}
\pgfdeclarelayer{invishack}
\pgfsetlayers{invishack,triangles,edges,main,toplayer}
\newenvironment{edges}{\begin{pgfonlayer}{edges}}{\end{pgfonlayer}}
\newenvironment{tris}{\begin{pgfonlayer}{triangles}}{\end{pgfonlayer}}
\newenvironment{ontop}{\begin{pgfonlayer}{toplayer}}{\end{pgfonlayer}}
\newenvironment{invis}{\begin{pgfonlayer}{invishack}}{\end{pgfonlayer}}
\newcommand{\comp}[1]{\overline{#1}}
\newcommand{\robG}{Let $G$ be a robust graph}
\newcommand{\robrem}{Let $G$ be a robust graph with no reducible set}
\newcommand{\rlpart}[1]{part (\ref{rl:#1}) of Lemma~\ref{lem:redlem}}
\newcommand{\st}{\colon\,}
\tikzstyle{vertex}=[inner sep = 0pt, minimum width=4pt, fill=black, shape=circle]
\tikzstyle{squarevert}=[inner sep = 0pt, minimum width=4pt, minimum height=4pt, fill=white, shape=rectangle, draw=black, thick]
\tikzstyle{trivert}=[inner sep = 0pt, minimum width=6pt, minimum height=6pt, fill=gray!50!white, shape=regular polygon,regular polygon sides=3, draw=black, thick]
\tikzstyle{pentvert}=[inner sep = 0pt, minimum width=6pt, minimum height=6pt, fill=gray!15!white, shape=regular polygon,regular polygon sides=5, draw=black, thick]

\newcommand{\illSX}[2]{Now $#1$ is reducible using the following sets $\sey$ and $X$, illustrated in #2:}
\newcommand{\tillSX}[2]{Now $#1$ is reducible using the following sets $\sey$ and $X$, with $\sey$ illustrated in #2:}
\newcommand{\iso}{\cong}

\newcommand{\wke}{\mathsf{WKE}}
\newcommand{\ke}{\mathsf{KE}}
\newcommand{\iswke}[1]{$#1 \in \wke$}
\newcommand{\iske}[1]{$#1 \in \ke$}

\newcommand{\caze}[2]{\textbf{Case {#1}:} \textit{#2}}
\newcommand{\gpoint}[2]{\node[style=vertex, label=#1:$#2$]}
\newcommand{\vpoint}[2]{\node[style=squarevert, label=#1:$#2$]}
\newcommand{\bpoint}[1]{\gpoint{below}{#1}}
\newcommand{\apoint}[1]{\gpoint{above}{#1}}
\newcommand{\lpoint}[1]{\gpoint{left}{#1}}
\newcommand{\rpoint}[1]{\gpoint{right}{#1}}
\newcommand{\bvoint}[1]{\vpoint{below}{#1}}
\newcommand{\avoint}[1]{\vpoint{above}{#1}}
\newcommand{\lvoint}[1]{\vpoint{left}{#1}}
\newcommand{\rvoint}[1]{\vpoint{right}{#1}}

\renewcommand{\subset}{\subseteq}

\newcommand{\join}{\vee}

\DeclareMathOperator{\mad}{Mad}

\newtheorem{proposition}{Proposition}[section]
\newtheorem{lemma}[proposition]{Lemma}
\newtheorem{theorem}[proposition]{Theorem}
\newtheorem{corollary}[proposition]{Corollary}
\newtheorem*{robustrecap}{Lemma~\ref{lem:robust}}
\newtheorem*{fakemain}{Theorem~\ref{thm:mainthm}}
\newtheorem{conjecture}[proposition]{Conjecture}
\theoremstyle{definition}
\newtheorem{definition}[proposition]{Definition}

\newcommand{\sizeof}[1]{\left\lvert{#1}\right\rvert}
\newcommand{\floor}[1]{\left\lfloor{#1}\right\rfloor}

\newcommand{\tee}{\mathcal{T}}
\newcommand{\sey}{\mathcal{S}}

\title{Tuza's Conjecture for
 Graphs with Maximum Average Degree less than $7$}
\author{Gregory J.~Puleo}
\ead{puleo@illinois.edu}
\address{Department of Mathematics\\University of Illinois at Urbana-Champaign\\1409 W.~Green St, Urbana, IL 61801}
\address{Now at: Coordinated Science Lab\\University of Illinois at Urbana-Champaign\\1308 W.~Main Street, Urbana, IL 61801}
\date{\today}
\begin{document}
\begin{abstract} Tuza's~Conjecture states that if a graph $G$ does not
  contain more than $k$ edge-disjoint triangles, then some set of at
  most $2k$ edges meets all triangles of $G$. We prove
  Tuza's~Conjecture for all graphs $G$ having no subgraph with average
  degree at least $7$.  As a key tool in the proof, we introduce a
  notion of \emph{reducible sets} for Tuza's Conjecture; these are
  substructures which cannot occur in a minimal counterexample to
  Tuza's~Conjecture. We also introduce \emph{weak K\"onig--Egerv\'ary
    graphs}, a generalization of the well-studied K\"onig--Egerv\'ary
  graphs.
\end{abstract}
\maketitle
\section{Introduction}
Suppose that we wish to make a graph $G$ triangle-free by deleting a
small number of edges. An obvious obstruction is the presence of a
large family of edge-disjoint triangles: we must delete one edge from
each such triangle. On the other hand, deleting all edges from a
maximal family of edge-disjoint triangles clearly destroys all
triangles in $G$.  Let $\nu(G)$ denote the maximum size of a set of
edge-disjoint triangles in $G$, and let $\tau(G)$ denote the minimum
size of an edge set $Y$ such that $G-Y$ is triangle-free. We have just
argued that $\nu(G) \leq \tau(G) \leq 3\nu(G)$. Clearly the lower
bound is sharp, with equality in many instances, such as when all
blocks are triangles.  The desire to obtain a sharp upper bound
motivates the following conjecture:
\begin{conjecture}[Tuza's Conjecture~\cite{TuzaProc,Tuza}]
  $\tau(G) \leq 2\nu(G)$ for all graphs $G$.
\end{conjecture}
Any graph whose blocks are all isomorphic to $K_4$ achieves equality
in the upper bound, as observed by Tuza~\cite{Tuza}.

Tuza's~Conjecture has been studied by many authors. The best general
upper bound on $\tau(G)$ in terms of $\nu(G)$ is due to
Haxell~\cite{Haxell}, who showed that $\tau(G) \leq 2.87\nu(G)$ for
all graphs $G$.

Other authors have pursued the conjecture by showing that the desired
bound $\tau(G) \leq 2\nu(G)$ holds for certain special classes of
graphs.  Tuza~\cite{Tuza} showed that his conjecture holds for all
planar graphs and for all $K_6$-free chordal graphs.
Aparna~Lakshmanan, Bujt\'as, and Tuza~\cite{LBT} generalized the
result for planar graphs by showing that the conjecture holds for all
``triangle-$3$-colorable'' graphs, a class containing all
$4$-colorable graphs. Another generalization is due to
Krivelevich~\cite{Krivelevich}, who showed that Tuza's~Conjecture
holds for all graphs having no $K_{3,3}$-subdivision.  The result on
planar graphs was extended in a different direction by Haxell,
Kostochka, and Thomass\'e~\cite{SashaK4}, who proved that when $G$ is
a $K_4$-free planar graph, the stronger inequality $\tau(G) \leq
\frac{3}{2}\nu(G)$ holds.

Krivelevich~\cite{Krivelevich} proved that a version of
Tuza's~Conjecture holds when $\tau$ or $\nu$ is replaced by its
fractional relaxation $\tau^*$ or $\nu^*$, where instead of asking for
a set of edges $Y$ or a set of edge-disjoint triangles $\tee$, one
instead asks for a \emph{weight function} on the edges or the
triangles of $G$, subject to constraints on the weight function which
model the original constraints on $Y$ and $\tee$. Chapuy, DeVos,
McDonald, Mohar, and Schiede~\cite{CDMMS} improved Krivelevich's bound of
$\tau(G) \leq 2\nu^*(G)$ to the stronger bound $\tau(G) \leq 2\nu^*(G) -
\frac{1}{\sqrt{6}}\sqrt{\nu^*(G)}$, and proved that this bound is
tight. Chapuy, DeVos, McDonald, Mohar, and Schiede also extended
Tuza's result on planar graphs, as well as Haxell's result, to the
context of weighted graphs. Krivelevich's result was also extended by
Haxell, Kostochka, and Thomass\'e~\cite{SashaStability}, who proved a
stability theorem: if $\tau^*(G) \geq 2\nu^*(G) - x$, then $G$
contains a family of pairwise edge-disjoint subgraphs consisting of
$\nu(G) - \floor{10x}$ copies of $K_4$ as well as $\floor{10x}$
triangles.

Haxell and R\"odl~\cite{HaxellRodl} showed that if $G$ is an
$n$-vertex graph and $\nu^*(G)$ is the fractional relaxation of
$\nu(G)$, then $\nu^*(G) - \nu(G) = o(n^2)$. As observed by
Yuster~\cite{Yuster}, this result together with Krivelevich's result imply $\tau(G) \leq
2\nu(G) + o(n^2)$; thus, Tuza's~Conjecture is asymptotically true for
graphs containing a quadratic-sized family of edge-disjoint
triangles. Such graphs are dense; instead, we study the conjecture on
sparse graphs.

An important measure of sparseness is the \emph{maximum
  average degree} of a graph, denoted $\mad(G)$ and defined
by
\[ \mad(G) = \max\left\{ \frac{2\sizeof{E(H)}}{\sizeof{V(H)}} \st H \subset
G \right\}. \] In this paper, we apply the discharging method to
prove the following theorem:
\begin{theorem}\label{thm:mainthm}
  If $\mad(G) < 7$, then $\tau(G) \leq 2\nu(G)$.
\end{theorem}
\noindent To our knowledge, this is the first application of the
discharging method to Tuza's~Conjecture, as well as the first
verification of the conjecture for a class of graphs defined only by a
sparsity condition, rather than more complex structural conditions. As
we discuss in Section~\ref{sec:consequences},
Theorem~\ref{thm:mainthm} further generalizes the earlier results
concerning planar graphs and $K_{3,3}$-subdivision-free graphs.

In Section~\ref{sec:main} we introduce definitions and give the
discharging argument used to prove Theorem~\ref{thm:mainthm}, modulo
two lemmas whose proofs occupy most of the paper. The key definition
in Section~\ref{sec:main} is that of a \emph{reducible set}, a
particular substructure that cannot occur in a smallest counterexample
to Tuza's~Conjecture. Essentially, a reducible set represents a
``local solution'' to the optimization problem posed by
Tuza's~Conjecture.

The definition of a reducible set for Tuza's~Conjecture is perhaps the
main new idea of the paper. While we use discharging to prove the
existence of reducible sets, we hope that later work will be able to
use these reducible sets in extremal arguments which may not involve
discharging at all.


In Sections~\ref{sec:wke}--\ref{sec:subsume} we prove the two lemmas
stated in Section~\ref{sec:main}.  In Section~\ref{sec:wke} we
introduce \emph{weak K\"onig--Egerv\'ary graphs}, which we use heavily
in our reducibility proofs. In Section~\ref{sec:low} we discuss the
behavior of low-degree vertices in graphs with no reducible set.

The results in Sections~\ref{sec:main}--\ref{sec:low} are sufficient
to prove a weaker result than Theorem~\ref{thm:mainthm}. Using these
results, we can show that Tuza's~Conjecture holds for all graphs $G$
with $\mad(G) < 25/4$, a threshold which still suffices for many of
the desired applications. In Section~\ref{sec:625} we pause and give a
short proof of Tuza's~Conjecture for graphs $G$ with $\mad(G) < 25/4$.

In Section~\ref{sec:subsume} we explore the relation of
\emph{subsumption}, which plays a prominent role in the discharging
rule of Section~\ref{sec:main} and allows us to push the maximum
average degree threshold up to $7$. We again explore the behavior of
this relation in graphs with no reducible set.
\section{Definitions and Proof Summary}\label{sec:main}
When $G$ is a graph and $W \subset V(G)$, we write $G[W]$ for the
subgraph of $G$ induced by the vertices in $W$. When $V_0 \subset
V(G)$, we write $N(V_0)$ for $\bigcup_{v \in V_0}N(v)$, and when $U
\subset V(G)$, we write $N_U(V_0)$ for $N(V_0) \cap U$. Similarly,
$d_U(V_0)$ denotes $\sizeof{N_U(V_0)}$. For $v \in V(G)$, we write $N[v]$
for the closed neighborhood $N(v) \cup \{v\}$. We write $K_n^-$ to denote the
complete graph on $n$ vertices with any edge removed. When the
graph $G$ is understood and $k$ is a nonnegative integer, we say that
a vertex of $G$ is a \emph{$k$-vertex} if its degree in $G$ is exactly
$k$, a \emph{$k^+$-vertex} if its degree is at least $k$, or a
\emph{$k^-$-vertex} if its degree is at most $k$.

While Tuza's~Conjecture involves two combinatorial optimization
parameters, it can be also viewed as a single combinatorial
optimization problem: in this problem, the goal is to simultaneously
find a set $\tee$ of edge-disjoint triangles and an edge set $Y$ such
that $G-Y$ is triangle-free and such that $\sizeof{Y} \leq
2\sizeof{\tee}$.

The requirement that $G-Y$ be triangle-free is a global
requirement. We replace this global problem with a local problem:
fixing a vertex set $V_0$, we seek a set $\sey$ of edge-disjoint
triangles and an edge set $X$ such that $G-X$ has no triangle
containing a vertex of $V_0$ and such that $\sizeof{X} \leq
2\sizeof{\sey}$. The rough idea is to remove the vertex set $V_0$,
solve the ``global'' problem in the resulting subgraph, and then
combine the subgraph solution with the ``local solution'' to solve the
global problem in the original graph.  The main difficulty in combining
solutions this way is the requirement that the final set of triangles
be edge-disjoint; carelessly combining sets of triangles will violate
this requirement. The definition of a reducible set is tailored to
overcome this difficulty:
\begin{definition}\label{def:reducible}
  When $\sey$ is a set of triangles, an \emph{$\sey$-edge} is an edge
  of some triangle in $\sey$. A nonempty set $V_0 \subset V(G)$ is
  \emph{reducible} if there exist a set $\sey$ of edge-disjoint
  triangles in $G$ and set $X$ of edges in $G$ such that the following
  conditions hold:
  \begin{enumerate}[(i)]
  \item $\sizeof{X} \leq 2\sizeof{\sey}$; 
  \item $G-X$ has no triangle containing a vertex of $V_0$; and
  \item $X$ contains every $\sey$-edge whose endpoints are both outside $V_0$.
  \end{enumerate}
  When $V_0$, $\sey$, and $X$ satisfy the definition above, we say
  that $V_0$ is \emph{reducible using $\sey$ and $X$}.
\end{definition}
Note that Tuza's~Conjecture holds for $G$ if and only if the entire
vertex set $V(G)$ is reducible. However, if $G$ is a \emph{minimal}
counterexample to Tuza's~Conjecture, then $G$ has no reducible set of
any size:
\begin{lemma}\label{lem:reducible-wins}
  Let $G$ be a graph, and let $V_0 \subset V(G)$ be reducible using
  $\sey$ and $X$. Let $G' = (G-X) - V_0$. If $\tau(G') \leq 2\nu(G')$,
  then $\tau(G) \leq 2\nu(G)$.
\end{lemma}
\begin{proof}
  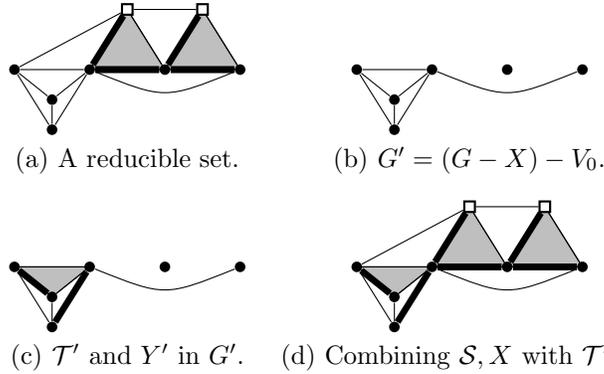
\begin{figure}
    \centering
    \begin{tabular}{cc}
      \begin{tikzpicture}
        \begin{scope}[yscale=.8]
        \avoint{} (v1) at (0cm, 0cm) {};
        \avoint{} (v2) at (1cm, 0cm) {};
        \apoint{} (w1) at (-1.5cm, -1cm) {};
        \apoint{} (w2) at (-.5cm, -1cm) {};
        \apoint{} (w3) at (.5cm, -1cm) {};
        \apoint{} (w4) at (1.5cm, -1cm) {};
        \apoint{} (z1) at (-1cm, -1.5cm) {};
        \apoint{} (z2) at (-1cm, -2cm) {};
        \draw (v1) -- (v2);
        \draw (w2) .. controls (.5cm, -1.5cm) .. (w4);
        \draw (w1) -- (w2) -- (w3) -- (w4);
        \draw (v1) -- (w1); \draw (v1) -- (w2); \draw (v1) -- (w3);
        \draw (v2) -- (w3); \draw (v2) -- (w4);
        \draw (w1) -- (z1); \draw (w1) -- (z2);
        \draw (w2) -- (z1); \draw (w2) -- (z2);
        \draw (z1) -- (z2);
        \begin{tris}
          \draw[triangle1] (v1.center) -- (w2.center) -- (w3.center) -- cycle;
          \draw[triangle1] (v2.center) -- (w3.center) -- (w4.center) -- cycle;          
        \end{tris}
        \draw[xedge] (w2) -- (w3); \draw[xedge] (w3) -- (w4);
        \draw[xedge] (w2) -- (w3); \draw[xedge] (w3) -- (w4);
        \draw[xedge] (v1) -- (w2); \draw[xedge] (v2) -- (w3);
        \end{scope}
      \end{tikzpicture}      
      &\begin{tikzpicture}
        \begin{scope}[yscale=.8]
        \apoint{} (w1) at (-1.5cm, -1cm) {};
        \apoint{} (w2) at (-.5cm, -1cm) {};
        \apoint{} (w3) at (.5cm, -1cm) {};
        \apoint{} (w4) at (1.5cm, -1cm) {};
        \apoint{} (z1) at (-1cm, -1.5cm) {};
        \apoint{} (z2) at (-1cm, -2cm) {};
        \draw (w1) -- (w2);
        \draw (w1) -- (z1); \draw (w1) -- (z2);
        \draw (w2) -- (z1); \draw (w2) -- (z2);
        \draw (w2) .. controls (.5cm, -1.5cm) .. (w4);
        \draw (z1) -- (z2);
        \begin{invis}
        \draw[white] (.5cm, 0cm) ellipse (1cm and .5cm);          
        \end{invis}
        \end{scope}
        \end{tikzpicture}\\
      (a) A reducible set. & (b) $G' = (G-X) - V_0$. \\
      \begin{tikzpicture}
        \begin{scope}[yscale=.8]
        \apoint{} (w1) at (-1.5cm, -1cm) {};
        \apoint{} (w2) at (-.5cm, -1cm) {};
        \apoint{} (w3) at (.5cm, -1cm) {};
        \apoint{} (w4) at (1.5cm, -1cm) {};
        \apoint{} (z1) at (-1cm, -1.5cm) {};
        \apoint{} (z2) at (-1cm, -2cm) {};
        \draw (w1) -- (w2);
        \draw (w1) -- (z1); \draw (w1) -- (z2);
        \draw (w2) -- (z1); \draw (w2) -- (z2);
        \draw (w2) .. controls (.5cm, -1.5cm) .. (w4);
        \draw (z1) -- (z2);
        \begin{invis}
        \draw[white] (.5cm, 0cm) ellipse (1cm and .5cm);          
        \end{invis}
        \begin{tris}
          \draw[triangle1] (w1.center) -- (w2.center) -- (z1.center) -- cycle;
        \end{tris}
        \draw[xedge] (w1) -- (z1);
        \draw[xedge] (w2) -- (z2);
        \end{scope}
        \end{tikzpicture} & 
\begin{tikzpicture}
        \begin{scope}[yscale=.8]
        \avoint{} (v1) at (0cm, 0cm) {};
        \avoint{} (v2) at (1cm, 0cm) {};
        \apoint{} (w1) at (-1.5cm, -1cm) {};
        \apoint{} (w2) at (-.5cm, -1cm) {};
        \apoint{} (w3) at (.5cm, -1cm) {};
        \apoint{} (w4) at (1.5cm, -1cm) {};
        \apoint{} (z1) at (-1cm, -1.5cm) {};
        \apoint{} (z2) at (-1cm, -2cm) {};
        \draw (v1) -- (v2);
        \draw (w1) -- (w2) -- (w3) -- (w4);
        \draw (w2) .. controls (.5cm, -1.5cm) .. (w4);
        \draw (v1) -- (w1); \draw (v1) -- (w2); \draw (v1) -- (w3);
        \draw (v2) -- (w3); \draw (v2) -- (w4);
        \draw (w1) -- (z1); \draw (w1) -- (z2);
        \draw (w2) -- (z1); \draw (w2) -- (z2);
        \draw (z1) -- (z2);
        \begin{invis}
          \draw[white] (.5cm, 0cm) ellipse (1cm and .5cm);
        \end{invis}
        \begin{tris}
          \draw[triangle1] (v1.center) -- (w2.center) -- (w3.center) -- cycle;
          \draw[triangle1] (v2.center) -- (w3.center) -- (w4.center) -- cycle;
          \draw[triangle1] (w1.center) -- (w2.center) -- (z1.center) -- cycle;
        \end{tris}
        \draw[xedge] (w2) -- (w3); \draw[xedge] (w3) -- (w4);
        \draw[xedge] (w2) -- (w3); \draw[xedge] (w3) -- (w4);
        \draw[xedge] (v1) -- (w2); \draw[xedge] (v2) -- (w3);
        \draw[xedge] (w1) -- (z1);
        \draw[xedge] (w2) -- (z2);
        \end{scope}
        \end{tikzpicture}\\
      (c) $\tee'$ and $Y'$ in $G'$. & (d) Combining $\sey, X$ with $\tee', Y'$.
    \end{tabular}
    \caption{Using a reducible set. Shaded triangles represent $\sey$
      and $\tee'$; thick edges represent $X$ and $Y'$; square white
      vertices represent $V_0$.}
    \label{fig:reducible}
  \end{figure}
  Let $\tee'$ be a largest set of edge-disjoint triangles in $G'$, and
  let $Y'$ be a smallest set of edges such that $G' - Y'$ is
  triangle-free; by hypothesis, $\sizeof{Y'} \leq
  2\sizeof{\tee'}$. Let $\tee = \tee' \cup \sey$ and $Y = Y' \cup X$.
  (The process is illustrated in Figure~\ref{fig:reducible}.)  We show
  that $\tee$ is a set of edge-disjoint triangles in $G$, that $G - Y$
  is triangle-free, and that $\sizeof{Y} \leq 2\sizeof{\tee}$, thus
  establishing the desired conclusion. The third condition is
  immediate from $\sizeof{Y'} \leq 2\sizeof{\tee'}$ and $\sizeof{X}
  \leq 2\sizeof{\sey}$.

  To show that the triangles in $\tee$ are pairwise edge-disjoint, it
  suffices to show that no $\sey$-edge is a $\tee'$-edge. This holds
  because every $\tee'$-edge is contained in $(G - X) - V_0$, while
  every $\sey$-edge is incident to $V_0$ or contained in $X$, by
  Condition~(iii) of Definition~\ref{def:reducible}.

  Next we show that $G - Y$ is triangle-free. This holds because any
  triangle $T$ in $G$ satisfies one of the following three conditions:
  \begin{enumerate}[(1)]
  \item $T$ is contained in $(G-X) - V_0$; or
  \item $T$ contains a vertex of $V_0$; or
  \item $T$ contains an edge of $X$.
  \end{enumerate}
  Triangles of the first type meet $Y'$, by hypothesis; 
  triangles of the second type meet $X$, by Condition~(ii) of
  Definition~\ref{def:reducible}; triangles of the third type
  meet $X$ by definition.
\end{proof}
Our strategy for applying Lemma~\ref{lem:reducible-wins} is typical of
discharging arguments: we show that various possible substructures of
a graph $G$ imply the existence of a reducible set, and we show that
every graph with average degree less than $7$ has one of these
substructures. For more background on the discharging method, see
\cite{discharging}.

To give the list of forbidden substructures, a few new definitions are
needed:
\begin{definition}
  A graph $G$ is \emph{robust} if for every $v \in
  V(G)$, every component of $G[N(v)]$ has order at least $5$.
\end{definition}
\noindent If $G$ is robust, then $\delta(G) \geq 5$. Also, $G[N(v)]$ is
connected whenever $d(v) < 10$.

First-time readers may wish to skip the rest of this section, as well
as Section~\ref{sec:subsume}, as the weaker version of the result
given in Section~\ref{sec:625} has a much simpler proof. The remaining
definitions are not needed for the proof of the weaker result.
\begin{definition}
  A vertex $u$ \emph{subsumes} a vertex $v$ if $N[u] \supset N[v]$.
\end{definition}
\noindent Equivalently, $u$ subsumes $v$ if $u$ is a dominating vertex in $G[N(v)]$.
\begin{definition}
  A $6$-vertex $v$ is \emph{thin} if $\overline{G[N(v)]}$ contains
  a matching of size $3$.
\end{definition}
The full list of forbidden substructures is given by the following two
lemmas; the proof of the second lemma will occupy most of the paper.
For each part of the second lemma, we indicate which later results
imply that part of the lemma.  
\begin{lemma}\label{lem:robust}
  If $G$ is a minimal counterexample to Tuza's~Conjecture, then $G$
  is robust.
\end{lemma}
\begin{lemma}\label{lem:redlem}
  If $G$ is robust and has no reducible set, then the following conditions hold.
  \begin{enumerate}[(a)]
  \item Every $6^-$-vertex $v \in V(G)$ satisfies $\Delta(\comp{G[N(v)]}) \leq 1$ and $\sizeof{E(\comp{G[N(v)]})} \neq 2$. (Proposition~\ref{prop:comp-matching})\label{rl:match}
  \item The $6^-$-vertices of $G$ form an independent set. (Proposition~\ref{prop:red-pair-6})\label{rl:ind6}
  \item No $7$-vertex subsumes any $6$-vertex. (Lemma~\ref{lem:6-dom-7})\label{rl:7sub6}
  \item No $7$-vertex is adjacent to any thin $6$-vertex. (Lemma~\ref{lem:6-perf-7})\label{rl:7adjL6}
  \item No $8^-$-vertex subsumes any $5$-vertex. (Lemma~\ref{lem:few-8-nbors} and Proposition~\ref{prop:red-pair-6})\label{rl:8sub5}
  \item Every $9$-vertex subsumes at most three $6^-$-vertices, and a
    $9$-vertex subsuming three $6^-$-vertices is adjacent to exactly
    three $6^-$-vertices. (Lemma~\ref{lem:9conq})\label{rl:9conq}
  \item Every $10^+$-vertex $v$ that subsumes some $6^-$-vertex has at
    most $d(v)-6$ neighbors that are
    $6^-$-vertices. (Lemma~\ref{lem:bigconq})\label{rl:bigconq}
  \item Every vertex $v$ with $d(v) \in \{7,8,9\}$ has at most $d(v)-4$ neighbors that are $6^-$-vertices. (Corollary~\ref{cor:few6minus})\label{rl:minus4}
  \end{enumerate}
\end{lemma}
Postponing the proof of Lemmas~\ref{lem:robust} and~\ref{lem:redlem},
we now give the proof of the main theorem.
\begin{lemma}\label{lem:has-reducible}
  Every robust graph with average degree less than $7$ has a reducible set.
\end{lemma}
\begin{proof}
  Assuming that $G$ has no reducible set, we use the discharging
  method to show that $G$ has average degree at least $7$. Give every
  vertex $v$ initial charge $d(v)$. We apply the following discharging
  rule:
  \begin{itemize}
  \item Every $5$-vertex takes charge $2/3$ from each vertex subsuming it;
  \item Every thin $6$-vertex takes charge $1/6$ from each neighbor;
  \item Every non-thin $6$-vertex takes charge $1/4$ from each vertex subsuming it.
  \end{itemize}
  We claim that every vertex has final charge at least $7$, yielding
  average degree at least $7$ in $G$.

  First we consider the $6^-$-vertices. By \rlpart{ind6}, no two such
  vertices are adjacent, so no $6^-$-vertex loses any charge when the
  discharging rule is applied. Thus we only need to check that each
  type of $6^-$-vertex gains enough charge to reach $7$.  There are no
  $4^-$-vertices, since $G$ is robust.  By \rlpart{match}, every
  $5$-vertex is subsumed by at least three vertices, and hence gains
  at least $2$.  Every thin $6$-vertex gains exactly $1$, for final
  charge $7$.  By \rlpart{match}, the neighborhood of a non-thin
  $6$-vertex $v$ lacks at most one edge; hence $v$ is subsumed by at
  least four vertices, and gains at least $1$.

  Now we consider the higher-degree vertices. Each $7$-vertex starts
  with charge $7$ and loses none, since it does not subsume any $5$- or
  $6$-vertices and is not adjacent to any thin $6$-vertex, by parts
  (\ref{rl:7sub6})--(\ref{rl:8sub5}) of Lemma~\ref{lem:redlem}.
  
  Next, let $v$ be an $8$-vertex. By \rlpart{8sub5}, $v$ does not
  subsume any $5$-vertices. By \rlpart{minus4}, $v$ is adjacent to at
  most four $6^-$-vertices; hence, $v$ loses at most $4(1/4)$ charge,
  yielding final charge at least $7$.

  Now, let $v$ be a $9$-vertex. By \rlpart{minus4}, $v$ has at most
  five $6^-$-neighbors in total.  Hence, if $v$ subsumes at most two
  $6^-$-vertices, then $v$ loses at most $2(2/3) + 3(1/6)$ charge,
  yielding final charge greater than $7$. On the other hand, if $v$
  subsumes three $6^-$-vertices, then by \rlpart{9conq} we see that
  $v$ is adjacent to exactly those three $6^-$-vertices, so $v$ loses
  exactly $3(2/3)$ charge, yielding final charge at least $7$.

  Finally, let $v$ be a $k$-vertex with $k \geq 10$. If $v$ subsumes
  no $6^-$-vertex, then $v$ loses charge at most $k/6$, which yields
  final charge at least $7$ since $k - k/6 \geq 7$. If $v$ subsumes
  some $6^-$-vertex, then at most $k-6$ neighbors of $v$ are
  $6^-$-vertices, by \rlpart{bigconq}. Thus $v$ loses at most
  $2(k-6)/3$, which yields final charge at least $7$ since $k -
  2(k-6)/3 \geq 7$. Hence all vertices have final charge at least $7$,
  yielding average degree at least $7$.
\end{proof}
\begin{fakemain}
  If $\mad(G) < 7$, then $\tau(G) \leq 2\nu(G)$.  
\end{fakemain}
\begin{proof}
  If the claim fails, let $G$ be a minimal counterexample. Since
  $\mad(G) < 7$, any proper subgraph $G'$ also satisfies $\mad(G') <
  7$, so $\tau(G') \leq 2\nu(G')$ by the minimality of $G$. Thus, $G$
  is a minimal counterexample to Tuza's~Conjecture among all graphs.
  By Lemma~\ref{lem:robust}, $G$ is robust, so by
  Lemma~\ref{lem:has-reducible}, $G$ has a reducible set.  Now
  Lemma~\ref{lem:reducible-wins} yields $\tau(G) \leq 2\nu(G)$,
  contradicting the choice of $G$ as a counterexample.
\end{proof}
In the next section we explore some applications of
Theorem~\ref{thm:mainthm} and its supporting lemmas. The remainder of
the paper will then be devoted to proving Lemmas~\ref{lem:robust}
and~\ref{lem:redlem}.
\section{Consequences}\label{sec:consequences}
Theorem~\ref{thm:mainthm} extends several earlier results on Tuza's
Conjecture, yielding these results (or stronger versions thereof) as
consequences.  Tuza~\cite{Tuza} proved that the conjecture holds for
planar graphs.  The following corollary of Euler's~Formula extends the
result to toroidal graphs, which are the graphs of genus at most $1$.
The conjecture was not previously known to hold for such graphs.
\begin{proposition}\label{prop:euler}
  If $G$ is an $n$-vertex graph of genus $\gamma$ with $m$ edges, then
  $m \leq 3(n - 2 + 2\gamma)$. In particular, $G$ has average degree
  at most $6 + \frac{12(\gamma-1)}{n}$.
\end{proposition}
\begin{corollary}\label{cor:toroidal}
  If $G$ is toroidal, then $\tau(G) \leq 2\nu(G)$.
\end{corollary}
For higher genus, we obtain a finitization result.
\begin{proposition}\label{prop:finite}
  For any fixed $\gamma$ with $\gamma \geq 2$, if $\tau(G) \leq
  2\nu(G)$ for all graphs $G$ of genus at most $\gamma$ with
  $\sizeof{V(G)} \leq 12(\gamma-1)$, then $\tau(G) \leq 2\nu(G)$ for
  all graphs $G$ of genus at most $\gamma$.
\end{proposition}
\begin{proof}
  Suppose not; let $G$ be a minimal counterexample among the graphs of
  genus at most $\gamma$. All proper subgraphs $G'$ also have genus at
  most $\gamma$, so they satisfy $\tau(G') \leq 2\nu(G')$, by the
  minimality of $G$.  By hypothesis, $\sizeof{V(G)} > 12(\gamma-1)$,
  so $G$ has average degree less than $7$. By Lemma~\ref{lem:robust},
  $G$ is robust, so by Lemma~\ref{lem:has-reducible}, $G$ has a
  reducible set.  Thus $\tau(G) \leq 2\nu(G)$, by
  Lemma~\ref{lem:reducible-wins}, contradicting the choice of $G$ as a
  counterexample.
\end{proof}
For the case $\gamma=2$, we performed an exhaustive computer search to
verify the hypothesis of Proposition~\ref{prop:finite}. By using
Lemma~\ref{lem:redlem}, we avoid explicitly checking Tuza's~Conjecture
on graphs that can be shown to have reducible sets. Using the
isomorph-free generation program \texttt{geng}~\cite{Nauty} with a
custom pruning function designed to recognize forbidden configurations
(a) and (b) in Lemma~\ref{lem:redlem}, we identified a set of only
$5299$ graphs that contains any smallest counterexample of genus $2$.
(A database of these graphs, and tools for verifying the database,
can be found at \url{http://www.math.uiuc.edu/~puleo/tuzaverify.tar.gz})
For higher $\gamma$, this computational approach quickly becomes
intractible, even with such optimizations.

Krivelevich~\cite{Krivelevich} proved Tuza's~Conjecture for graphs
with no $K_{3,3}$-subdivision. We obtain the same result using
Theorem~\ref{thm:mainthm}. Our proof relies on a theorem of
Wagner~(\cite{Wagner}, described in \cite{Thomas}).
\begin{definition}
  Let $G_1$ and $G_2$ be graphs. A \emph{$k$-sum} of $G_1$ and $G_2$ is any
  graph obtained by identifying the vertices of a $k$-clique in $G_1$ with
  a $k$-clique in $G_2$ and then possibly deleting some edges of the merged
  $k$-clique. (In particular, a $0$-sum is a disjoint union.)
\end{definition}
The following theorem was originally stated for graphs having no $K_{3,3}$-minor,
rather than graphs having no $K_{3,3}$-subdivision. However, since $\Delta(K_{3,3}) \leq 3$,
a graph is $K_{3,3}$-minor-free if and only if it is $K_{3,3}$-subdivision-free (see Proposition~1.7.4 of \cite{diestel}).
Thus, we restate the theorem in terms of subdivisions.
\begin{theorem}[Wagner~\cite{Wagner}] 
  Any graph with no $K_{3,3}$-subdivision can be obtained by a sequence of
  $0$-,$1$-,or $2$-sums starting from planar graphs and $K_5$.
\end{theorem}
\begin{corollary}\label{cor:3n5}
  If $G$ is a graph with $n$ vertices and no $K_{3,3}$-subdivision, then $\sizeof{E(G)} \leq 3n-5$.
\end{corollary}
\begin{proof}
  If $G$ is planar or $G = K_5$, then the conclusion holds. It
  suffices to show that if $G_1$ and $G_2$ are graphs satisfying the
  bound, then any $0$-, $1$-, or $2$-sum of $G_1$ and $G_2$ also
  satisfies the bound. This follows by straightforward algebra.  In
  particular, for a $j$-sum of $G_1$ and $G_2$ with $j \in \{0,1,2\}$
  and $n_i = \sizeof{V(G_i)}$,
  \begin{align*}
    \sizeof{E(G)} &\leq \sizeof{E(G_1)} + \sizeof{E(G_2)} - {j \choose 2} \\
    &\leq 3n_1 + 3n_2 - 10 - {j \choose 2} \\
    &= 3n-5 + \left(3j - {j \choose 2} - 5\right) \leq 3n-5.
  \end{align*}
  (Recall that ${0 \choose 2} = {1 \choose 2} = 0$.)
\end{proof}
\begin{theorem}[Krivelevich]\label{thm:k33}
  If $G$ is a graph with no $K_{3,3}$-subdivision, then $\tau(G) \leq 2\nu(G)$.
\end{theorem}
\begin{proof}
  Since all subgraphs of $G$ also have no $K_{3,3}$-subdivision,
  Corollary~\ref{cor:3n5} implies $\mad(G) < 6$.
\end{proof}
Aparna~Lakshmanan, Bujt\'as, and Tuza~\cite{LBT} proved that if $G$ is
$4$-colorable, then $\tau(G) \leq 2\nu(G)$. This implies that
Tuza's~Conjecture holds for all graphs with no $K_5$-minor, since (as
Wagner~\cite{Wagner} showed) the Four-Color~Theorem implies that all
graphs with no $K_5$-minor are $4$-colorable. Using a theorem of
Mader~\cite{Mader} together with Theorem~\ref{thm:mainthm}, we instead
obtain the result for graphs with no $K_5$-subdivision:
\begin{theorem}\label{thm:k5sub}
  If $G$ is a graph with no $K_5$-subdivision, then $\tau(G) \leq 2\nu(G)$.
\end{theorem}
\begin{proof}
  Since $G$ has no $K_5$-subdivision, the number of edges in $G$ is at
  most $3\sizeof{V(G)}-6$, as proved by Mader~\cite{Mader}.  All
  subgraphs of $G$ are $K_5$-subdivision-free, so $\mad(G) < 6$.
\end{proof}
Tuza's~Conjecture was not previously known to hold for graphs with
a $K_5$-minor but no $K_5$-subdivision.
\section{Weak K\"onig--Egerv\'ary Graphs}\label{sec:wke}
A graph $H$ is a \emph{K\"onig--Egerv\'ary graph} if $\alpha'(H) =
\beta(H)$, where $\alpha'(H)$ is the matching number and $\beta(H)$ is
the vertex cover number. The concept was introduced by Deming~\cite{Deming}; see
also Kayll~\cite{Kayll}. Let $\ke$ denote the class of
K\"onig--Egerv\'ary graphs. The K\"onig--Egerv\'ary Theorem~\cite{egervary,konig} says that
if $H$ is bipartite, then \iske{H}. We weaken this definition, obtaining
a larger class of graphs which will help streamline our reducibility proofs.
\begin{definition}\label{def:wke}
  The graph $H$ is a \emph{weak K\"onig--Egerv\'ary graph} if
  $H$ has a matching $M$ and a vertex set $Q \subset V(H)$ such that
  $\sizeof{Q} \leq \sizeof{M}$ and $Q$ is a vertex cover in $H-M$.
  Let $\wke$ denote the class of weak K\"onig--Egerv\'ary graphs.
  We say that a pair $(M,Q)$ as above \emph{witnesses} \iswke{H}.
\end{definition}
Observe that $\ke \subset \wke$: if \iske{H}, then $(M,Q)$
witnesses \iswke{H}, where $M$ is any maximum matching and $Q$ is any
minimum vertex cover in $H$.

To relate weak K\"onig--Egerv\'ary graphs to reducible sets, we
introduce an edge version of reducibility. The intuitive explanation
is similar to the explanation of vertex reducibility: to say that the
edge set $E_0$ is reducible is to say that we can remove $E_0$, solve
the global problem in the resulting subgraph, and then combine the
global solution for the subgraph with a carefully-chosen local
solution in order to obtain a global solution for the original graph.
\begin{definition}\label{def:edge-reducible}
  A nonempty edge set $E_0 \subset E(G)$ is \emph{reducible} if there
  exist a set $\sey$ of edge-disjoint triangles and a set $X$ of edges
  of $G$ such that the following conditions hold:
  \begin{enumerate}[(i)]
  \item $\sizeof{X} \leq 2\sizeof{\sey}$; and 
  \item $G-X$ has no triangle containing an edge of $E_0$; and
  \item $X$ contains every $\sey$-edge that is not in $E_0$.
  \end{enumerate}
  When $E_0$, $\sey$, and $X$ satisfy the definition above, we say
  that $E_0$ is \emph{reducible using $\sey$ and $X$}.
\end{definition}
An analogue of Lemma~\ref{lem:reducible-wins} holds for reducible edge sets.
The proof is essentially the same, so we do not repeat it here:
\begin{lemma}\label{lem:edge-reducible}
  Let $G$ be a graph, and let $E_0 \subset E(G)$ be reducible using
  $\sey$ and $X$. Let $G' = (G - X) - E_0$.  If $\tau(G') \leq
  2\nu(G')$, then $\tau(G) \leq 2\nu(G)$.
\end{lemma}
\begin{lemma}\label{lem:wke}
  Let $v \in V(G)$, and let $G_0$ be a nonempty union of components of
  $G[N(v)]$. If \iswke{G_0}, then $G$ has a reducible set of
  edges. Also, if \iswke{G[N(v)]}, then $\{v\}$ is reducible.
\end{lemma}
\begin{proof}
  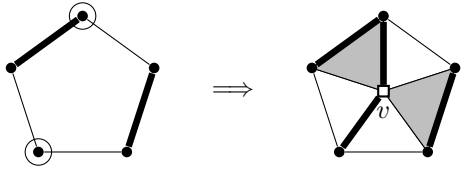
\begin{figure}
    \centering
    \begin{tikzpicture}
      \begin{scope}[xshift=-2cm, rotate=90]
        \apoint{} (a1) at (0: 1cm) {};
        \apoint{} (a2) at (72: 1cm) {};
        \apoint{} (a3) at (144: 1cm) {};
        \apoint{} (a4) at (216: 1cm) {};
        \apoint{} (a5) at (288: 1cm) {};
        \begin{edges}
        \draw (a1) -- (a2) -- (a3) -- (a4) -- (a5) -- (a1);
        \draw[medge] (a1) -- (a2);
        \draw[medge] (a4) -- (a5);
        \qvert{a1};
        \qvert{a3};
        \end{edges}
      \end{scope}
      \node at (0cm, 0cm) {$\Longrightarrow$};
       \begin{scope}[xshift=2cm, rotate=90]
         \bvoint{v} (v) at (0cm, 0cm) {};
         \apoint{} (b1) at (0: 1cm) {};
         \apoint{} (b2) at (72: 1cm) {};
         \apoint{} (b3) at (144: 1cm) {};
         \apoint{} (b4) at (216: 1cm) {};
         \apoint{} (b5) at (288: 1cm) {};
         \begin{edges}
         \foreach \i in {1,2,3,4,5}
         {
           \draw (v) -- (b\i);
         }
        \draw (b1) -- (b2) -- (b3) -- (b4) -- (b5) -- (b1);
        \draw[triangle1] (b1.center) -- (b2.center) -- (v.center) -- cycle;
        \draw[triangle1] (b4.center) -- (b5.center) -- (v.center) -- cycle;
        \draw[xedge] (v) -- (b3);
        \draw[xedge] (v) -- (b1);
        \draw[xedge] (b1) -- (b2);
        \draw[xedge] (b4) -- (b5);           
         \end{edges}
      \end{scope}
    \end{tikzpicture}
    \caption{Transforming $(M,Q)$ to $(\sey,X)$.}
    \label{fig:wke}
  \end{figure}
  Take any pair $(M,Q)$ witnessing \iswke{G_0}. Define $E_0$, $\sey$, and $X$ as follows:
  \begin{align*}
    E_0 &= \{ vw \st w \in G_0 \}; \\
    \sey &= \{vuw \st uw \in M\}; \\
    X &= M \cup \{vx \st x \in Q\}.
  \end{align*}
  Figure~\ref{fig:wke} illustrates the definition of $\sey$ and $X$;
  in the figure, thick edges represent $M$ and $X$, circled vertices
  represent $Q$, and shaded triangles represent $\sey$.

  Since $M$ is a matching, the triangles in $\sey$ are pairwise
  edge-disjoint. We claim that $E_0$ is reducible using $\sey$ and
  $X$. Verifying each condition of Definition~\ref{def:edge-reducible}
  in turn:
  \begin{enumerate}[(i)]
  \item Clearly $\sizeof{X} \leq 2\sizeof{\sey}$, since $\sizeof{Q} \leq \sizeof{M}$.
  \item Any triangle of $G$ containing an edge of $E_0$ has the form
    $vxy$, where $xy \in E(G_0)$.  Since $Q$ is a vertex cover in
    $G_0-M$, either $xy \in M$ or one of its endpoints lies in
    $Q$. Thus $G-X$ has no such triangle.
  \item $X$ contains every $\sey$-edge that does not contain an edge
    of $E_0$, since $M \subset X$.
  \end{enumerate}
  For the second claim, we observe that when $E_0$ is defined as above
  and $G_0 = G[N(v)]$, the condition $e \in E_0$ is equivalent to the
  condition $v \in e$. Comparing Definition~\ref{def:edge-reducible}
  to Definition~\ref{def:reducible} , this shows that reducibility of
  $E_0$ is equivalent to reducibility of $\{v\}$.
\end{proof}
Since all bipartite graphs are weak K\"onig--Egerv\'ary graphs,
Lemma~\ref{lem:wke} extends a theorem of Aparna~Lakshmanan, Bujt\'as,
and Tuza~\cite{LBT}, who proved that if $G$ is odd-wheel-free (i.e.,
locally bipartite) then Tuza's Conjecture holds for $G$. We prove a
more precise extension of this result later in the section.

In the remainder of this section, we seek sufficient conditions for a
graph to be a weak K\"onig--Egerv\'ary~graph. Due to
Lemma~\ref{lem:wke}, these conditions yield restrictions on the vertex
neighborhoods in a minimum counterexample to Tuza's~Conjecture.

The first such result is an analogue of the
K\"onig--Egerv\'ary~Theorem: if $H$ has no odd cycle of length
greater than $3$, then \iswke{H}.  The proof relies on a
characterization of such graphs due to Hsu,
Ikura, and Nemhauser~\cite{HIN}.
\begin{theorem}[Hsu--Ikura--Nemhauser~\cite{HIN}]\label{thm:HIN}
  If $H$ is $2$-connected and has no odd cycle of length greater than
  $3$, then $H$ is either bipartite, isomorphic to $K_4$, or
  isomorphic to $K_2 \join \overline{K_r}$ for some $r \geq 1$.
\end{theorem}
\noindent We start by proving a natural consequence of the
K\"onig--Egerv\'ary~Theorem.
\begin{lemma}\label{lem:strongke}
  If $v$ is a vertex in a bipartite graph $H$, then $v$ lies in some
  minimum vertex cover of $H$ if and only if $v$ is covered by every
  maximum matching of $H$.
\end{lemma}
\begin{proof}
  If $v$ lies in some minimum vertex cover $Q$, then $Q-v$ is a vertex
  cover of $H-v$, so $\beta(H-v) \leq \beta(H) - 1$.  By the
  K\"onig--Egerv\'ary~Theorem, $\alpha'(H-v) \leq \alpha'(H) - 1$.
  This implies that $v$ is covered by every maximum matching of $H$.

  Now assume that $v$ is covered in every maximum matching of
  $H$. This yields $\alpha'(H-v) = \alpha'(H) - 1$.  By the
  K\"onig--Egerv\'ary~Theorem, $\beta(H-v) = \alpha'(H) - 1$.  Adding
  $v$ to a minimum vertex cover in $H-v$ yields the desired vertex
  cover.
\end{proof}
Next, we consider the nonbipartite graphs in Theorem~\ref{thm:HIN}, obtaining
a stronger version of the weak K\"onig--Egerv\'ary property:
\begin{lemma}\label{lem:missv}
  If $H$ is $2$-connected and has no odd cycle of length greater than
  $3$, then for every $v \in V(H)$, there is a vertex set $Q$ and a
  matching $M$ such that:
  \begin{enumerate}
  \item $\sizeof{Q} \leq \sizeof{M}$,
  \item $Q$ is a vertex cover in $H-M$, and
  \item Either $v \in Q$ or $v$ is in no edge of $M$.
  \end{enumerate}  
  In particular, \iswke{H}.
\end{lemma}
\begin{proof}
  If $H$ is bipartite, then the claim follows from
  Lemma~\ref{lem:strongke} together with the
  K\"onig--Egerv\'ary~Theorem: if $v$ is covered by every maximum
  matching, then any minimum vertex cover has the desired
  properties. Thus, we may assume $H$ is not bipartite.  By
  Theorem~\ref{thm:HIN}, it suffices to consider three cases:
  
  \caze{1}{$H \iso K_3$.} Write $V(H) = \{v, w_1, w_2\}$. Let $Q =
  \{v\}$ and let $M = \{w_1w_2\}$; the only edge of $H$ not covered by
  $v$ is $w_1w_2$.

  \caze{2}{$H \iso K_4$ or $H \iso K_2 \join \overline{K_2}$.} Either
  way, $H$ has a matching $M$ of size $2$. Let $Q$ be $v$ together
  with its mate in $M$; the only edge of $H$ not covered by $Q$ is the
  other edge in $M$.

  \caze{3}{$H \iso K_2 \join \overline{K_m}$ for $m \geq 3$.}  Let $Q$
  consist of the two vertices of maximum degree. If $v \in Q$, then
  let $M$ be any matching of size $2$.  Otherwise, $\alpha'(H-v) = 2$,
  so we can take $M$ to be any maximum matching in $H-v$.
\end{proof}
\begin{proposition}\label{prop:longodd}
  If $H$ has no odd cycle of length greater than $3$, then \iswke{H}.
\end{proposition}
\begin{proof}
  We use induction on $\sizeof{V(H)}$. If $\sizeof{V(H)} = 1$
  then clearly \iswke{H}. Now suppose that $\sizeof{V(H)} > 1$
  and the claim holds for all graphs with fewer vertices and
  no odd cycle of length greater than $3$.

  By the induction hypothesis, we may assume that $H$ is connected. On
  the other hand, if $H$ is $2$-connected, then \iswke{H}, by
  Lemma~\ref{lem:missv}. Thus we may assume that $H$ is connected but
  not $2$-connected, so $H$ has a leaf block $B$. Note that
  $\sizeof{V(B)} \geq 2$.

  Let $v$ be the cut vertex of $H$ contained in $B$. By
  Lemma~\ref{lem:missv}, $B$ has a matching $M_B$ and vertex cover
  $Q_B$ such that $\sizeof{Q_B} \leq \sizeof{M_B}$, $Q_B$ is a vertex
  cover in $B - M_B$, and either $v \in Q_B$ or $v$ is in no edge of
  $M_B$.

  Now define a subgraph $H'$ as follows: if $v \in Q_B$, let $H' =
  H-V(B)$; otherwise, let $H' = H - (V(B) - v)$. By the induction
  hypothesis, \iswke{H'}; let $M'$ and $Q'$ witness \iswke{H'}. Let $Q
  = Q' \cup Q_B$ and $M = M' \cup M_B$.  Whether or not $v \in Q_B$,
  we see that $M$ is a matching and that $Q$ is a vertex cover in
  $H-M$. Clearly $\sizeof{Q} \leq \sizeof{M}$, so \iswke{H}.
\end{proof}
\begin{corollary}\label{cor:deg4}
  If $H$ is connected and $\sizeof{V(H)} \leq 4$, then \iswke{H}.
\end{corollary}
\begin{corollary}\label{cor:small}  
  If $H$ is connected and $\alpha'(H) \leq 1$, then \iswke{H}.
  Also, if $H$ is connected, $\sizeof{V(H)} > 5$, and
  $\alpha'(H) = 2$, then \iswke{H}.
\end{corollary}
\begin{proof}
  If $H \notin \wke$, then $H$ contains a cycle of length at least
  $5$, which implies $\alpha'(H) \geq 2$. For the second claim,
  observe that any cycle of length at least $6$ contains a matching of
  size $3$, while if $C$ is a $5$-cycle in $H$, then there are
  adjacent vertices $v \in V(C)$ and $w \notin V(C)$, which yields the
  following matching of size $3$:
  
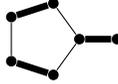
\begin{center}
    \begin{tikzpicture}
      \foreach \i in {0,...,4}
      {
        \apoint{} (z\i) at (72*\i : .5cm) {};
      }
      \apoint{} (y) at (1cm,0cm) {};
      \draw[medge] (z0) -- (y);
      \draw[medge] (z1) -- (z2);
      \draw[medge] (z3) -- (z4);
      \draw (z0) -- (z1);
      \draw (z2) -- (z3);
      \draw (z4) -- (z0);
    \end{tikzpicture}
  \end{center}
\end{proof}
We can now give a more specific generalization of the theorem of
Aparna~Lakshmanan, Bujt\'as, and Tuza~\cite{LBT} concerning
odd-wheel-free graphs:
\begin{corollary}\label{cor:wheel}
  If $G$ has no subgraph isomorphic to any wheel graph $W_n$ for
  $n$ odd with $n \geq 5$, then $\tau(G) \leq 2\nu(G)$.
\end{corollary}
\begin{proof}
  We use induction on $\sizeof{V(G)}$. When $\sizeof{V(G)} = 1$ the
  claim is obvious. When $\sizeof{V(G)} > 1$, take any vertex $v \in
  V(G)$.  The subgraph $G[N(v)]$ has no odd cycle of length greater
  than $3$, since that would yield a forbidden wheel in $G$. Hence
  \iswke{G[N(v)]}. By Proposition~\ref{prop:longodd}, $\{v\}$ is
  reducible. By the induction hypothesis and
  Lemma~\ref{lem:reducible-wins}, the claim holds.
\end{proof}
In particular, Corollary~\ref{cor:wheel} implies a recent result of
Aparna~Lakshmanan, Bujt\'as, and Tuza~\cite{LBT-perfect}, who
independently proved that if the triangle graph $\mathcal{T}(G)$ is
perfect, then $\tau(G) \leq 2\nu(G)$. See \cite{LBT-perfect} for the
definition of the triangle graph.

Recall that $G$ is \emph{robust} if for every $v \in V(G)$, every
component of $G[N(v)]$ has order at least $5$. In Section~\ref{sec:main}
we stated the following lemma, which now follows from the earlier
results of this section:
\begin{robustrecap}
  If $G$ is a minimal counterexample to Tuza's~Conjecture, then $G$
  is robust.  
\end{robustrecap}
\begin{proof}
  Follows immediately from Lemma~\ref{lem:edge-reducible},
  Lemma~\ref{lem:wke}, and Corollary~\ref{cor:deg4}.
\end{proof}
\begin{corollary}\label{cor:ind}
  Let $H$ be an $n$-vertex connected graph, where $n \geq 6$. If $H$
  has an independent set of size $n-3$, then \iswke{H}.
\end{corollary}
\begin{proof}
  If $\alpha'(H) \leq 2$, then \iswke{H}, by Corollary~\ref{cor:small}.
  Otherwise, $\alpha'(H) \geq 3$, and the complement in $V(H)$ of a
  maximum independent set is a vertex cover of size at most $3$.
  Thus \iske{H}.
\end{proof}
Finally, we give a sufficient condition for small graphs to be
weak K\"onig--Egerv\'ary graphs. (In fact, the condition is also
necessary, but we do not need the other direction, so we omit it.)
\begin{proposition}\label{prop:weak-comp-matching}
  Let $H$ be an $n$-vertex connected graph, where $n \in \{5,6\}$. If
  $\Delta(\comp{H}) > 1$, then \iswke{H}.
\end{proposition}
\begin{proof} Since $\Delta(\comp{H}) > 1$, we may take $u, z_1, z_2
  \in V(H)$ such that $uz_1, uz_2 \notin E(H)$. By
  Corollary~\ref{cor:small}, we may assume $\alpha'(H) = n - 3$, since $n
  \in \{5,6\}$. If $z_1z_2 \notin E(H)$, then $V(H) - \{u,z_1,z_2\}$
  is a vertex cover in $H$ having size $n - 3$, which implies
  \iske{H}.  Thus we may assume $z_1z_2 \in E(H)$. Also, if there is
  some maximum-size matching $M$ containing the edge $z_1z_2$, then
  $V(H) - \{u,z_1,z_2\}$ is a vertex cover of size $n-3$ in $H-M$,
  which implies \iswke{H}.

  \caze{1}{$n = 5$ and $\alpha'(H) = 2$}. Since no maximum-size
  matching contains $z_1z_2$, there are no edges in $H - \{z_1,z_2\}$,
  so $\{z_1,z_2\}$ is a vertex cover in $H$. Hence \iske{H}.

  \caze{2}{$n = 6$ and $\alpha'(H) = 3$}.  Let $H_0 = H -
  \{z_1,z_2\}$; since no maximum-size matching contains $z_1z_2$, we
  have $\alpha'(H_0) \leq 1$.  By Corollary~\ref{cor:ind}, if $H$ has
  an independent set of size $3$, then \iswke{H}. Thus we may
  assume that $\alpha(H) < 3$, which implies that $H_0$ is a graph on
  $4$ vertices such that $\alpha'(H_0) \leq 1$ and $\alpha(H_0) <
  3$. This is only possible if $H_0 \iso K_3 + K_1$, as illustrated in
  Figure~\ref{fig:k1k3}.\begin{figure} \centering
    \begin{tikzpicture}
      \apoint{z_1} (z1) at (0cm, 0cm) {};
      \apoint{z_2} (z2) at (1cm, 0cm) {};
      \qvert{z1};
      \qvert{z2};
      \draw (z1) -- (z2);
      \lpoint{y} (y) at (0cm, -1cm) {};
      \qvert{y};
      \lpoint{} (y1) at (-.5cm, -1.5cm) {};
      \lpoint{} (y2) at (.5cm, -1.5cm) {};
      \draw (y) -- (y1) -- (y2) -- (y);
      \draw[ultra thick] (y1) -- (y2) node[below,pos=.5] {$e_1$};
      \apoint{} (k1) at (1.5cm, -1cm) {};
      \draw[ultra thick] (z2) -- (k1) node[right,pos=.5] {$e_3$};
      \draw[ultra thick] (z1) -- (y) node[left, pos=.5] {$e_2$};
      \draw (y) -- (z1);
      \draw (k1) -- (z2);
      \draw[dashed] (-.75cm, -.75cm) -- (1.75cm, -.75cm) -- (1.75cm, -2cm) -- (-.75cm, -2cm) -- cycle;
      \node at (2cm, -1.375cm) {$H_0$};
    \end{tikzpicture}
    \caption{Finding a matching and vertex cover in $H$.}
    \label{fig:k1k3}
  \end{figure}
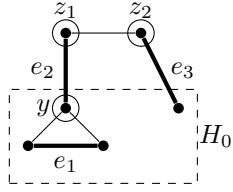

  It follows that if $M$ is any maximum-size matching in $H$, then one
  edge of $M$ must lie inside the $K_3$ component of $H_0$, one edge
  must join the $K_3$ component to $\{z_1,z_2\}$, and one edge must
  join the $K_1$ component to $\{z_1,z_2\}$. Fix some maximum-size
  matching $M$ and label its edges $e_1, e_2, e_3$ respectively.  Let
  $y$ be the vertex in the $K_3$ component not contained in $e_1$. The
  set $\{y,z_1,z_2\}$ is a vertex cover in $H-M$, so \iswke{H}.
\end{proof}
Using weak K\"onig--Egerv\'ary graphs, we have shown that any
minimum counterexample to Tuza's~Conjecture is robust (and therefore
has minimum degree at least $5$), and we have obtained strong
restrictions on the possible neighborhoods of any $5$- or $6$-
vertex.
\section{Low-Degree Vertices}\label{sec:low}
In this section, we will study the behavior of $6^-$-vertices in
graphs with no reducible set. The main result of the section is
Proposition~\ref{prop:red-pair-6}, which states that the
$6^-$-vertices form an independent set; this result is used heavily in
Section~\ref{sec:subsume}.

We first obtain a stronger version of
Proposition~\ref{prop:weak-comp-matching}, using the observation that
it is possible for $\{v\}$ to be reducible even though $G[N(v)] \notin
\wke$.
\begin{proposition}\label{prop:comp-matching}
  \robG.  If $v \in V(G)$ with $d(v) \leq 6$, then $\{v\}$ is
  reducible in $G$ if $\Delta(\comp{G[N(v)]}) > 1$.  Also, if $d(v)
  \leq 6$ and $\comp{G[N(v)]}$ has exactly $2$ edges, then $\{v\}$ is
  reducible.
\end{proposition}
\begin{proof}The first statement follows immediately from Lemma~\ref{lem:wke} and
  Proposition~\ref{prop:weak-comp-matching}.  For the second
  statement, we again split into cases according to $d(v)$. Let $H =
  G[N(v)]$, and let $w_1, \ldots, w_{d(v)}$ be the vertices of $H$,
  indexed so that $E(\comp{H}) = \{w_1w_2, w_3w_4\}$.

  \begin{figure}
    \centering
    \begin{tabular}{cc}
    \begin{tikzpicture}
      \avoint{v} (v) at (0cm, 1.5cm) {};
      \bpoint{w_5} (w5) at (270 : 1cm) {};
      \bpoint{w_3} (w3) at (342 : 1cm) {};
      \rpoint{w_1} (w1) at (54 : 1cm) {};
      \lpoint{w_4} (w4) at (126 : 1cm) {};
      \bpoint{w_2} (w2) at (198 : 1cm) {};
      \begin{tris}
        \draw[triangle1] (w5.center) -- (w2.center) -- (w3.center) -- cycle;
        \draw[triangle1] (w5.center) -- (w1.center) -- (w4.center) -- cycle;
        \draw[triangle1] (w1.center) -- (w3.center) ..controls (18:2cm) .. (v.center) -- cycle; 
        \draw[triangle1] (w4.center) -- (w2.center) ..controls (162:2cm) .. (v.center) -- cycle; 
      \end{tris}
      \draw (w2) -- (w3);
    \end{tikzpicture}&
    \begin{tikzpicture}
      \apoint{w_1} (w1) at (0cm, 0cm) {};
      \apoint{w_3} (w3) at (1cm, 0cm) {};
      \apoint{w_5} (w5) at (2cm, 0cm) {};
      \bpoint{w_4} (w4) at (0cm, -1cm) {};
      \bpoint{w_2} (w2) at (1cm, -1cm) {};
      \bpoint{w_6} (w6) at (2cm, -1cm) {};
      \begin{ontop}
        \rvoint{v} (v) at (3cm, -.5cm) {};
      \end{ontop}
      \begin{tris}
        \draw[triangle1] (v.center) -- (w1.center) -- (w4.center) -- cycle;
        \draw[triangle1] (v.center) -- (w3.center) -- (w2.center) -- cycle;
        \draw[triangle1] (v.center) -- (w5.center) -- (w6.center) -- cycle;
        \draw[triangle1] (w3.center) -- (w5.center) .. controls (1cm, .8cm) .. (w1.center) -- cycle;
        \draw[triangle1] (w2.center) .. controls ++(195:.5cm) and ++(-15:.5cm) .. (w4.center) .. controls ++(-45:1.5cm) and ++(-100:1.5cm) .. (w5.center) ..controls ++(-120:.5cm) and ++(30:.5cm) .. (w2.center) -- cycle;
      \end{tris}
      \begin{invis}
        \node[vertex, fill=white] at (-1cm, -.5cm) {};
      \end{invis}
      \draw (v.center) -- (w1.center) -- (w4.center) -- cycle;
      \draw (v.center) -- (w3.center) -- (w2.center) -- cycle;
      \draw (v.center) -- (w5.center) -- (w6.center) -- cycle;
    \end{tikzpicture}\\
    (a) Case 1: $d(v) = 5$.&
    (b) Case 2: $d(v) = 6$.
    \end{tabular}

    \caption{Triangles in Proposition~\ref{prop:comp-matching}.}
    \label{fig:comp-matching-5}
  \end{figure}
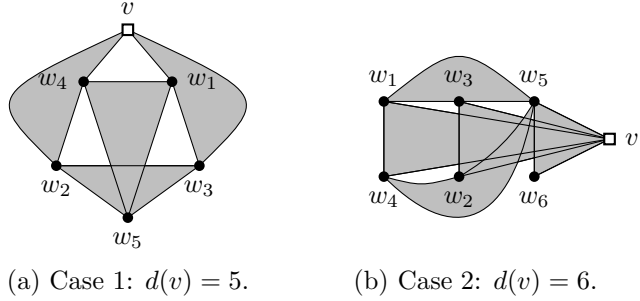
  \caze{1}{$d(v) = 5$}. Define $\sey$ and $X$ by
  \begin{align*}
    \sey &= \{vw_2v_4, vw_1w_3, w_1w_4w_5, w_2w_3w_5 \}, \\
    X &= E(H).
  \end{align*}
  The triangles in $\sey$ are illustrated in
  Figure~\ref{fig:comp-matching-5}(a).  We verify that $\{v\}$ is
  reducible using $\sey$ and $X$, verifying each condition of
  Definition~\ref{def:reducible}:
  \begin{enumerate}[(i)]
  \item $\sizeof{X} \leq 2\sizeof{\sey}$, since $\sizeof{E(H)} = 8$.
  \item Every triangle containing $v$ has its other two vertices in
    $H$, so $G-X$ has no triangle containing $v$.
  \item $X$ contains every $\sey$-edge not incident to $v$, since all
    such edges lie in $H$.
  \end{enumerate}

  \caze{2}{$d(v) = 6$}. Define $\sey$ and $X$ by
  \begin{align*}
    \sey &= \{  vw_1w_4, vw_2w_3, vw_5w_6, w_1w_3w_5, w_2w_4w_5 \}, \\
    X &= E(H - w_6) \cup \{ w_5w_6, vw_6 \}.
  \end{align*}
  The triangles in $\sey$ are illustrated in Figure~\ref{fig:comp-matching-5}(b).
  We verify that $\{v\}$ is reducible using $\sey$ and $X$, verifying
  each condition of Definition~\ref{def:reducible}:
  \begin{enumerate}[(i)]
  \item By construction, $\sizeof{X} \leq 2\sizeof{\sey}$.
  \item Since all edges of $H$ not incident to $w_6$ lie in $X$, any
    triangle of $G-X$ containing $v$ also contains $w_6$.  Since $vw_6
    \in X$, it follows that there is no such triangle.
  \item By inspection, $X$ contains every $\sey$-edge that is not incident to $v$.\qedhere
  \end{enumerate}
\end{proof}
In the rest of the paper, we will typically omit explicit
verifications of Conditions~(i) and (iii) of
Definition~\ref{def:reducible}, since they usually follow from a quick
inspection of the definitions.

Next we show that a robust graph with no reducible set can have no
edge joining ``low-degree'' vertices. The idea is simple: if $u$ and
$v$ are adjacent low-degree vertices and neither $\{u\}$ nor $\{v\}$
is reducible, then we have a lot of information about the structure of
$G[N(u)]$ and $G[N(v)]$, which will allow us to show that the set
$\{u,v\}$ is reducible.
\begin{proposition}\label{prop:red-pair-6}
  \robG.  If $uv \in E(G)$ with $d(u) \leq 6$
  and $d(v) \leq 6$, then one of $\{u\}$, $\{v\}$, or $\{u,v\}$
  is reducible in $G$.
\end{proposition}
\begin{proof} Without loss of generality, we may assume $d(u) \leq
  d(v)$.  Assuming that neither $\{u\}$ nor $\{v\}$ is reducible in
  $G$, we show that $\{u,v\}$ is reducible in $G$. Since neither
  $\{u\}$ nor $\{v\}$ is reducible,
  Proposition~\ref{prop:comp-matching} says that
  $\Delta(\comp{G[N(u)]}) \leq 1$ and $\Delta(\comp{G[N(v)]}) \leq 1$.
  Let $H = G[N(u) \cap N(v)]$.  Since $u,v \notin V(H)$ and $u$ has at
  most one non-neighbor in $G[N(v)]$, we have $d(v)-2 \leq
  \sizeof{V(H)} \leq d(u)-1$.  Also, $\Delta(\comp{H}) \leq 1$, since
  $\Delta(\comp{G[N(u)]}) \leq 1$.

  \caze{1}{$\sizeof{V(H)} = 3$.} In this case, $d(u) = d(v) = 5$ and
  $v$ is not a dominating vertex in $G[N(u)]$.  By
  Proposition~\ref{prop:comp-matching}, $G[N(u)]$ has precisely one
  non-edge, and likewise for $G[N(v)]$. Let $p$ the unique vertex in
  $N(u) - N[v]$, and let $q$ be the unique vertex in $N(v) - N[u]$;
  now $pv$ is the unique non-edge in $G[N(u)]$ and $qu$ is the unique
  non-edge in $G[N(v)]$. Write $V(H) = \{w_1,w_2,w_3\}$. Since $H
  \subset G[N(u)]$ and $pv$ is the unique non-edge in $G[N(u)]$, we
  have $H \iso K_3$.  \illSX{\{u,v\}}{Figure~\ref{fig:red-pair}(a)}
  \begin{align*}
    \sey &= \{uw_1w_2, vw_1w_3, upw_3, vqw_2\} \\
    X &= \{uv, vq, up, pw_3, qw_2\} \cup E(H).
  \end{align*}
  \begin{figure}
    \centering
    \begin{tikzpicture}
      \begin{scope}
        \bpoint{w_2} (w2) at (-60: 1.5cm) {};
        \apoint{w_3} (w3) at (60: 1.5cm) {};
        \bpoint{w_1} (w1) at (180: 1.5cm) {};
        \draw[xedge] (w1) -- (w2);
        \draw[xedge] (w2) -- (w3);
        \draw[xedge] (w3) -- (w1);
        \lvoint{u} (u) at (-2cm, 0cm) {};
        \lpoint{p} (p) at (-2cm, 1cm) {};
        \draw[xedge] (u) -- (p);
        \draw[xedge] (p) -- (w3);
        \avoint{v} (v) at (-.25cm, 0cm) {};
        \rpoint{q} (q) at (.25cm, 0cm) {};
        \draw[xedge] (v) -- (q);
        \draw[xedge] (q) -- (w2);
        \draw[xedge] (u) .. controls (-1.125cm, .5cm) .. (v);
        \begin{tris}
          \draw[triangle1] (u.center) ..controls ++(270:1cm) and ++(180:1cm) .. (w2.center) -- (w1.center) -- cycle;
          \draw[triangle1] (v.center) -- (w3.center) -- (w1.center) -- cycle;
          \draw[triangle1] (u.center) ..controls ++(90:1.25cm) and ++(210:.5cm) .. (w3.center) -- (p.center) -- cycle;
          \draw[triangle1] (v.center) -- (w2.center) -- (q.center) -- cycle;
        \end{tris}
        \node at (-.5cm, -2.5cm) {(a) $\sey, X$ in Case 1.};
      \end{scope}
      \begin{scope}[xshift=5cm]
        \begin{ontop}
          \bpoint{w_2} (w2) at (-60: 1.5cm) {};
          \apoint{w_3} (w3) at (60: 1.5cm) {};
          \bpoint{w_1} (w1) at (180: 1.5cm) {};
          \rpoint{w_4} (w4) at (2cm, 0cm) {};          
          \lvoint{u} (u) at (-2cm, 0cm) {};
          \lpoint{p} (p) at (-2cm, 1cm) {};
          \avoint{v} (v) at (-.25cm, 0cm) {};
          \rpoint{q} (q) at (.25cm, 0cm) {};
        \end{ontop}
        \draw[xedge] (w1) -- (w2);
        \draw[xedge] (w2) -- (w3);
        \draw[xedge] (w3) -- (w1);
        \draw[xedge] (w2) -- (w4);
        \draw[xedge] (w3) -- (w4);
        \draw[xedge] (u) -- (p);
        \draw[xedge] (p) -- (w3);
        \draw[xedge] (v) -- (q);
        \draw[xedge] (q) -- (w2);
        \draw[xedge] (u) .. controls (-1.125cm, .5cm) .. (v);
        \draw[xedge] (w1) .. controls (.75cm, -.5cm) .. (w4);  
        \draw (u.center) ..controls ++(90:1.25cm) and ++(210:.5cm) .. (w3.center) -- (p.center) -- cycle;
        \begin{tris}
          \draw[triangle1] (u.center) ..controls ++(270:1cm) and ++(180:1cm) .. (w2.center) -- (w1.center) -- cycle;
          \draw[triangle1] (v.center) -- (w3.center) -- (w1.center) -- cycle;
          \draw[triangle1] (u.center) ..controls ++(90:1.25cm) and ++(210:.5cm) .. (w3.center) -- (p.center) -- cycle;
          \draw[triangle1] (v.center) -- (w2.center) -- (q.center) -- cycle;
          \draw[triangle1] (w2.center) -- (w3.center) -- (w4.center) -- cycle;
          \draw[triangle1] (u.center) .. controls ++(60:2cm) and ++(90:3.5cm) .. (w4) .. controls ++(90:3.5cm) and ++(120:2cm) .. (v.center) .. controls (-1.125cm, .5cm) .. (u) -- cycle;
        \end{tris}
        \node at (0cm, -2.5cm) {(b) Largest possible $\sey, X$ in Case 2.};
      \end{scope}
    \end{tikzpicture}      
    \caption{Triangles and edges in Proposition~\ref{prop:red-pair-6}.}
    \label{fig:red-pair}
  \end{figure}
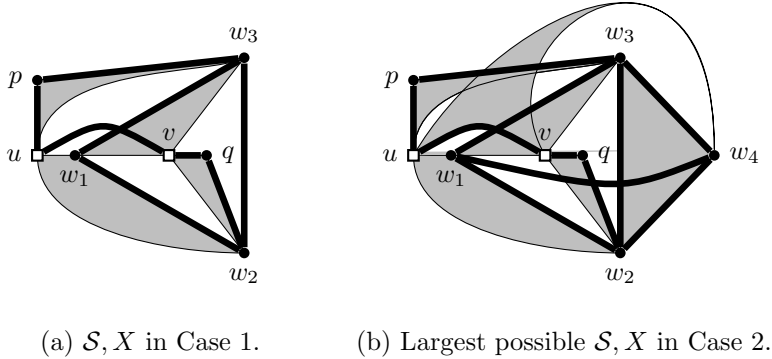%
  We quickly check Condition~(ii) of
  Definition~\ref{def:reducible}. Let $T$ be a triangle in $G-X$
  containing a vertex of $\{u,v\}$, say the vertex $u$. Since $E(H)
  \subset X$, at most one vertex of $H$ lies in $T$, so $T$ must
  contain a vertex in $\{v,p,q\}$. Since $uq \notin E(H)$ and $\{uv,
  up\} \subset X$, no such triangle exists. If $T$ instead contains
  $v$, a similar argument holds.

  \caze{2}{$\sizeof{V(H)} = 4$.} Since $\Delta(\comp{H}) \leq 1$, $H$
  contains incident edges $w_1w_2$ and $w_1w_3$; let $w_4$ be the
  remaining vertex of $H$. We build a set $\sey$ of triangles and a
  set $X$ of edges step-by-step as follows: initially, $\sey =
  \{uw_1w_2, vw_1w_3, uvw_4\}$ and $X = E(H) \cup \{uv\}$. Note that
  initially, $2\sizeof{\sey} - \sizeof{X} \geq -1$, with equality
  holding if and only if $H \iso K_4$. We augment $\sey$ and $X$
  according to the following rules:
  \begin{itemize}
  \item If there exists $p \in N(u) - N[v]$, then add the 
    triangle $upw_3$ to $\sey$ and add the edges $pu, pw_3$ to $X$.
  \item If there exists $q \in N(v) - N[u]$, then add the 
    triangle $vqw_2$ to $\sey$ and add the edges $qv, qw_2$ to $X$.
  \item If $H \iso K_4$, add the triangle $w_2w_3w_4$ to $\sey$.
  \end{itemize}
  Figure~\ref{fig:red-pair}(b) shows the $\sey$ and $X$ obtained when
  $p$ and $q$ both exist and $H \iso K_4$.  Note that if $p$ exists,
  then $p$ is the unique vertex of $N(u)-N[v]$, since $v$ has at most
  one non-neighbor in $G[N(u)]$; likewise for $q$.  In all cases, we
  end with $\sizeof{X} \leq 2\sizeof{\sey}$. The verification of
  Condition~(ii) is similar to Case~1.

  \caze{3}{$\sizeof{V(H)}=5$}. In this case, $d(u)=d(v)=6$ and $N[u]=N[v]$.
  Since $\Delta(\comp{H}) \leq 1$, $H$ contains a subgraph $H'$ isomorphic to $C_4$,
  with vertices $a,b,c,d$ in order. Let $w$ be the remaining vertex of $H$.
  \illSX{\{u,v\}}{Figure~\ref{fig:red-pair-6-cycle}}
  \begin{align*}
    \sey &= \{ uab, ucd, vbc, vad, uvw \}; \\
    X &= \{ uw, vw, ua, ub, vc, vd \} \cup E(H').
  \end{align*}%
  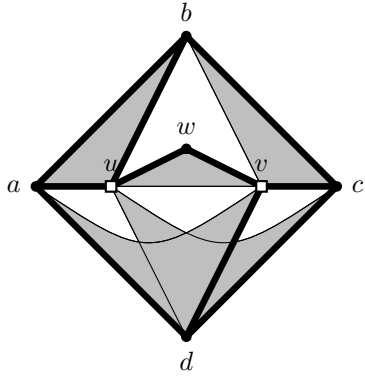
\begin{figure}
    \centering
    \begin{tikzpicture}
      \lpoint{a} (a) at (-2cm, 0cm) {};
      \apoint{b} (b) at (0cm, 2cm) {};
      \rpoint{c} (c) at (2cm, 0cm) {};
      \bpoint{d} (d) at (0cm, -2cm) {};
      \vpoint{above}{u} (u) at (-1cm, 0cm) {};
      \vpoint{above}{v} (v) at (1cm, 0cm) {};
      \draw (u) -- (v);
      \apoint{w} (w) at (0, .5cm) {};
      \begin{edges}
        \draw (u.center) -- (a.center) -- (b.center) -- cycle;
        \draw (v.center) -- (b.center) -- (c.center) -- cycle;
        \draw (u.center) -- (v.center) -- (w.center) -- cycle;
        \draw (u.center) .. controls (.5cm, -1cm) .. (c.center) -- (d.center) -- cycle;
        \draw (v.center) .. controls (-.5cm, -1cm) .. (a.center) -- (d.center) -- cycle;
        \draw[xedge] (a.center) -- (b.center) -- (c.center) -- (d.center) -- cycle;
        \draw[xedge] (u.center) -- (w.center);
        \draw[xedge] (v.center) -- (w.center);
        \draw[xedge] (u.center) -- (a.center);
        \draw[xedge] (u.center) -- (b.center);
        \draw[xedge] (v.center) -- (c.center);
        \draw[xedge] (v.center) -- (d.center);
      \end{edges}
      \begin{tris}
        \draw[triangle1] (u.center) -- (a.center) -- (b.center) -- cycle;
        \draw[triangle1] (v.center) -- (b.center) -- (c.center) -- cycle;
        \draw[triangle1] (u.center) -- (v.center) -- (w.center) -- cycle;
        \draw[triangle1] (u.center) .. controls (.5cm, -1cm) .. (c.center) -- (d.center) -- cycle;
        \draw[triangle1] (v.center) .. controls (-.5cm, -1cm) .. (a.center) -- (d.center) -- cycle;
      \end{tris}
    \end{tikzpicture}
    \caption{$\sey,X$ in Case~3 of Proposition~\ref{prop:red-pair-6}.}
    \label{fig:red-pair-6-cycle}
  \end{figure}%
  We again check Condition~(ii) of Definition~\ref{def:reducible}.
  Any triangle containing the edge $uv$ is of the form $uvz$, where $z
  \in V(H)$. For every such $z$, either $uz \in X$ or $vz \in X$;
  hence $G-X$ has no triangle containing $uv$.  Now suppose $T$ is a
  triangle containing $u$ but not $v$. Clearly $a,b,w \notin T$; hence
  $T = ucd$, but $cd \in X$. Hence $G-X$ has no triangle containing
  $u$. A similar argument holds for $v$.
\end{proof}
\begin{corollary}\label{cor:few6minus}
  \robG. If $v$ is a $k$-vertex with $k \in \{7,8,9\}$ and $v$ has
  more than $d(v)-4$ neighbors that are $6^-$-vertices, then $G$ 
  has a reducible set.
\end{corollary}
\begin{proof}
  By Proposition~\ref{prop:red-pair-6}, we may assume that the
  $6^-$-neighbors of $v$ form an independent set in $G[N(v)]$. Since
  $G$ is robust, $G[N(v)]$ is connected. By Corollary~\ref{cor:ind},
  if $G[N(v)]$ has an independent set of size $d(v)-3$, then
  \iswke{G[N(v)]}. Thus $\{v\}$ is reducible, by Lemma~\ref{lem:wke}.
\end{proof}
\section{A Weaker Result: $\mad(G) < 25/4$}\label{sec:625}
We now have sufficient tools to prove the following theorem, which is
weaker than Theorem~\ref{thm:mainthm} but still strong enough for many
of the applications in Section~\ref{sec:consequences}.  In particular,
this theorem is strong enough to imply Corollary~\ref{cor:toroidal} on
toroidal graphs, Theorem~\ref{thm:k33} on $K_{3,3}$-subdivision-free graphs,
and Theorem~\ref{thm:k5sub} on $K_5$-subdivision-free graphs.
\begin{theorem}\label{thm:weakthm}
  If $\mad(G) < 25/4$, then $\tau(G) \leq 2\nu(G)$.
\end{theorem}
\begin{proof}
  Assuming that $G$ has no reducible set, we use the discharging method
  to show that $G$ has average degree at least $25/4$. Give every vertex $v$
  initial charge $d(v)$. We apply the following discharging rule:
  \begin{itemize}
  \item Every $6^-$-vertex takes charge $1/4$ from every neighbor.
  \end{itemize}
  We claim that every vertex has final charge at least $25/4$, yielding
  average degree at least $25/4$ in $G$.

  First we consider the $6^-$-vertices. By Lemma~\ref{lem:robust}, $G$
  is robust, so $\delta(G) \geq 5$, and by Proposition~\ref{prop:red-pair-6},
  the $6^-$-vertices form an independent set. Hence all $5$-vertices
  end with charge $25/4$, and all $6$-vertices end with charge $30/4$.

  Next we consider the $k$-vertices for $k \in \{7,8,9\}$. By
  Corollary~\ref{cor:few6minus}, if $v$ is such a vertex, then $v$ has
  at most $k-4$ neighbors that are $6^-$-vertices.  Hence $v$ has
  final charge at least $3k/4+1$. Since $k \geq 7$, this implies that
  $v$ has final charge at least $25/4$.

  Finally, we consider the $10^+$-vertices. Such vertices have final
  charge at least $3k/4$, which is at least $25/4$, as desired.
\end{proof}
In the remaining section, we will improve the bound $\mad(G) < 25/4$ to $\mad(G) < 7$.
\section{Subsumption and Related Bounds}\label{sec:subsume}
Recall the following definitions from Section~\ref{sec:main}:
\begin{definition}
  A vertex $u$ \emph{subsumes} a vertex $v$ if $N[u] \supset N[v]$.
\end{definition}
\begin{definition}
  A $6$-vertex $v$ is \emph{thin} if $\overline{G[N(v)]}$ contains
  a matching of size $3$.
\end{definition}
The motivation for these definitions is as follows: when $u$ subsumes
a $6^-$-vertex $v$, having $d(v)-1$ neighbors of $v$ in $G[N(u)]$
leads to better bounds on the number of $6^-$-neighbors of $u$.  Thus,
in the discharging rule, such a vertex $u$ can give away a lot of
charge to the vertices it subsumes, since not many other
$6^-$-neighbors will place demands on it. Conversely, if $u$ subsumes
no $6^-$-vertex, then the bounds on the number of $6^-$-neighbors are
weaker, but since $u$ does not subsume its neighbors, they need not
demand much charge from $u$.
\begin{lemma}\label{lem:bigconq}
  \robrem.  If a $10^+$-vertex $v$ subsumes a $6^-$-vertex $w$, then
  at most $d(v)-6$ neighbors of $v$ are $6^-$-vertices.
\end{lemma}
\begin{proof}
  Assume to the contrary that at least $d(v)-5$ neighbors of $v$ are
  $6^-$-vertices. We obtain a contradiction by proving
  \iswke{G[N(v)]}, implying that $\{v\}$ is reducible. Let $A$ be the
  set of $6^-$-neighbors of $v$, and let $B = N(v) - A$; note that
  $\sizeof{A} \geq d(v) - 5 \geq 5 \geq \sizeof{B}$. Since $G$ has no
  reducible set, Proposition~\ref{prop:red-pair-6} implies that $A$ is
  an independent set and that $N(w) - \{v\} \subset B$. Since $A$ is
  independent, $B$ is a vertex cover in $G[N(v)]$.

  By Proposition~\ref{prop:comp-matching},
  $\Delta(\comp{G[N(a)]}) \leq 1$ for all $a \in A$; in particular,
  since $v \in N(a)$, we have $d_{G[N(a)]}(v) \geq d(a) - 2$. Since
  $d_{G[N(a)]}(v) = \sizeof{N(a) \cap N(v)} = d_{G[N(v)]}(a)$, we have
  $d_{G[N(v)]}(a) \geq d(a) - 2$. Since $A$ is independent and each
  $d(a) \geq 5$, this implies $d_B(a) \geq 3$ for all $a \in
  A$. Similarly, $d_B(w) \geq 4$, since $v$ is a dominating vertex in
  $G[N(w)]$ and $d(w) \geq 5$.

  We first argue that $\alpha'(G[N(v)]) \geq 4$ by greedily
  constructing a matching of size $4$. Let $a_1, a_2, a_3$ be distinct
  elements of $A-w$. Since each $d_B(a_i) \geq 3$, for each $i$ we may
  choose $b_i \in N_B(a_i)$ distinct from all earlier $b_i$. Since
  $d_B(w) \geq 4$, we can take $b' \in B - \{b_1,b_2,b_3\}$. Now
  $\{a_1b_1, a_2b_2, a_3b_3, wb'\}$ is the desired matching of size
  $4$. If $\sizeof{B} = 4$, then this implies \iske{G[N(v)]}; thus we
  may assume $\sizeof{B} = 5$. If $d_A(b) = 0$ for some $b \in B$,
  then $B-b$ is a vertex cover in $G[N(v)]$, which again implies
  \iske{G[N(v)]}; thus we may also assume $d_A(b) > 0$ for all $b \in
  B$.

  \caze{1}{$\sizeof{\bigcup_{z \in A-w}N_B(z)} = 3$.}  Let
  $z_1,z_2,z_3$ be distinct vertices in $A-w$, and let $b_1,b_2,b_3$
  be distinct vertices in $\bigcup_{z \in A-w}N_B(z)$.  Let $b' \in
  N_B(w) - \{b_1,b_2,b_3\}$. The set $B - b'$ is a vertex cover of
  size $4$ in $G[N(v)] - \{wb', z_1b_1, z_2b_2, z_3b_3\}$, so
  \iswke{G[N(v)]}. 

  \caze{2}{$\sizeof{\bigcup_{z \in A-w}N_B(z)} > 3$.} We verify
  Hall's~Condition for $B$. Take any $B_0 \subset B$.  If
  $\sizeof{B_0} > 2$, then $N_A(B_0) = A$, since each $a \in A$ has at
  most $2$ non-neighbors in $B$. If $\sizeof{B_0} = 1$, then
  $\sizeof{N_A(B_0)} \geq 1$ since $d_A(b) > 0$ for all $b \in
  B$. 

  Now suppose $\sizeof{B_0} = 2$. Since $d_B(w) \geq 4$, we have $w
  \in N_A(B_0)$.  For $z \in A-w$, if $z \notin N_A(B_0)$, then
  $N_B(z) = B-B_0$, since $d_B(z) \geq 3$. Since $\sizeof{\bigcup_{z
      \in A-w}N_B(z)} > 3$, the equality $N_B(z) = B-B_0$ cannot hold
  for all $z \in A-w$, so $\sizeof{N_A(B_0)} \geq 2$.  By
  Hall's~Theorem, $\alpha'(G[N(v)]) \geq 5$, so \iske{G}.
\end{proof}
When $d(v) = 9$ a similar statement holds, but more nuance is
required, since we are no longer guaranteed that $\sizeof{A} \geq \sizeof{B}$.
\begin{lemma}\label{lem:9conq}
  \robrem. Every $9$-vertex subsumes at most three $6^-$-vertices;
  furthermore, if equality holds, then it is adjacent to no other
  $6^-$-vertex.
\end{lemma}
\begin{proof}
  Let $v$ be a $9$-vertex subsuming $6^-$-vertices $w_1, w_2, w_3$.
  Suppose to the contrary that $v$ has another $6^-$-neighbor $w'$
  (possibly subsuming $w'$, possibly not).  Let $W =
  \{w_1,w_2,w_3,w'\}$, and let $V_0 = W \cup \{v\}$.  We show that
  $V_0$ is reducible, contradicting the hypothesis. By
  Proposition~\ref{prop:comp-matching}, we have
  $\Delta(\comp{G[N(w')]}) \leq 1$, since $G$ has no reducible set.
  Since $v \in N(w')$, this implies $\sizeof{N(w') - N[v]} \leq 1$. By
  the definition of subsumption, $N(w_i) \subset N[v]$ for each $i$.

  For convenience, let $H_i = G[N(w_i) \cap N(v)] = G[N(w_i) -
  \{v\}]$, and let $H' = G[N(w') \cap N(v)]$. By
  Proposition~\ref{prop:red-pair-6}, the $6^-$-vertices of $G$ form an
  independent set, so $V(H_i) \cap W = \emptyset$ for each $i$. We
  build a set $\sey$ of edge-disjoint triangles in several steps; the
  observation that $V(H_i) \cap W = \emptyset$ helps guarantee that
  $\sey$ is edge-disjoint.  The algorithm begins with $S = \emptyset$.

  (Figure~\ref{fig:sey9conq} illustrates the construction in the
  ``worst-case'' scenario where $V(H_1) = V(H_2) = V(H_3)$, where each
  $H_i \iso K_4^-$, and where $V(H')$ is a proper subset of $V(H_1)$.
  In general, it is possible that the subgraphs $H_i$ may have
  distinct vertex sets, but when they coincide we have less room to
  find edge-disjoint triangles. In the figure, dashed edges represent
  edges that are no longer available for use in $\sey$, since they
  were used in earlier triangles.)
  \begin{itemize}
  \item Since $\sizeof{V(H_1)} \geq 4$ and $\Delta(\comp{H_1}) \leq
    1$, we can find two disjoint edges $s_1s_2$ and $s_3s_4$ in
    $E(H_1)$.  Add the triangles $w_1s_1s_2$ and $w_1s_3s_4$ to
    $\sey$.
  \item Since $E(\overline{H_2}) \cup \{s_1s_2, s_3s_4\}$ is the union
    of two matchings and $\sizeof{V(H_2)} \geq 4$, we see that $H_2 -
    \{s_1s_2,s_3s_4\}$ is a graph on at least $3$ edges that is
    neither a star nor a triangle. Hence there are two disjoint edges
    $t_1t_2$ and $t_3t_4$ in $E(H_2) - \{s_1s_2, s_3s_4\}$. Add the
    triangles $w_2t_1t_2$ and $w_2t_3t_4$ to $\sey$.
  \item Since $\Delta(\comp{H_3}) \leq 1$ and $\sizeof{E(\comp{H_3})}
    \neq 2$, and since $\sizeof{V(H_3)} \geq 4$, we have
    $\sizeof{E(H_3)} \geq 5$.  Thus $H_3 - \{s_1s_2,
    s_3s_4, t_1t_2, t_3t_4\}$ still contains an edge $uu'$ and a
    vertex $r \notin \{u,u'\}$. Add the triangles $w_3uu'$ and
    $vw_3r$ to $\sey$.
  \item Since $\Delta(\comp{G[N(w')]}) \leq 1$, we have
    $\sizeof{V(H')} \geq 3$. Fix any vertex $r_1 \in V(H') - \{r\}$
    and add the triangle $vw'r_1$ to $\sey$, reaching seven
    triangles. Also, if $N(w') - N[w] \neq \emptyset$, let $p$ be the
    unique vertex in the difference. Note that $V(H') \subset N(p)
    \cap N(v)$, since otherwise $p$ would have two non-neighbors in
    $G[N(w')]$, contradicting $\Delta(\comp{G[N(w')]}) \leq 1$. Choose
    $r_2 \in V(H') - \{r, r_1\}$ and add the triangle $w'r_2p$ to
    $\sey$, reaching a total of eight triangles.
  \end{itemize}
  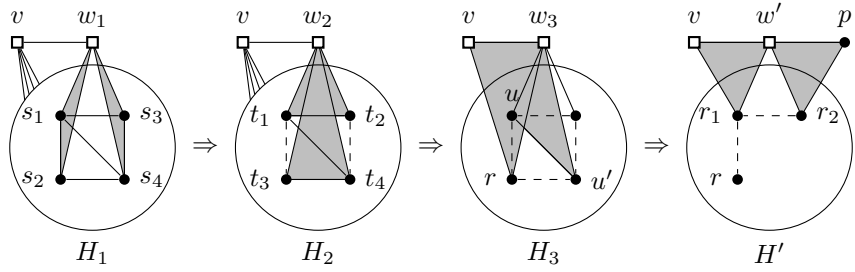
\begin{figure}\newcommand{\alv}{v\vphantom{w_3}}
    \centering
    \begin{tikzpicture}
      \begin{scope}[xshift=-5cm]
        \draw (0cm,0cm) circle (1.1cm);
        \begin{ontop}
          \avoint{\alv} (v) at (-1cm,1.4cm) {};
          \avoint{w_1} (w) at (0cm, 1.4cm) {};
        \end{ontop}
        \draw (v) -- (w);
        \draw (v) -- (130: 1.1cm);
        \draw (v) -- (135: 1.1cm);
        \draw (v) -- (140: 1.1cm);
        \draw (v) -- (145: 1.1cm);
        \rpoint{s_3} (z1) at (45:.6cm) {};
        \lpoint{s_1} (z2) at (135:.6cm) {};
        \lpoint{s_2} (z3) at (225:.6cm) {};
        \rpoint{s_4} (z4) at (315:.6cm) {};
        \draw (z1) -- (z2) -- (z3) -- (z4) -- (z1);
        \draw (z2) -- (z4);
        \begin{tris}
          \draw[triangle1] (z2.center) -- (z3.center) -- (w.center) -- cycle;
          \draw[triangle1] (z1.center) -- (z4.center) -- (w.center) -- cycle;
        \end{tris}
        \node at (270:1.4cm) {$H_1$};
      \end{scope}
      \begin{scope}[xshift=-2cm]
        \draw (0cm,0cm) circle (1.1cm);
        \begin{ontop}
          \avoint{w_2} (w) at (0cm, 1.4cm) {};
          \avoint{\alv} (v) at (-1cm,1.4cm) {};
        \end{ontop}
        \draw (v) -- (w);
        \draw (v) -- (130: 1.1cm);
        \draw (v) -- (135: 1.1cm);
        \draw (v) -- (140: 1.1cm);
        \draw (v) -- (145: 1.1cm);
        \rpoint{t_2} (z1) at (45:.6cm) {};
        \lpoint{t_1} (z2) at (135:.6cm) {};
        \lpoint{t_3} (z3) at (225:.6cm) {};
        \rpoint{t_4} (z4) at (315:.6cm) {};
        \draw (z1) -- (z2);
        \draw (z3) -- (z4);
        \draw[dashed] (z1) -- (z4);
        \draw[dashed] (z2) -- (z3);
        \draw (z2) -- (z4);
        \begin{tris}
          \draw[triangle1] (z1.center) -- (z2.center) -- (w.center) -- cycle;
          \draw[triangle1] (z3.center) -- (z4.center) -- (w.center) -- cycle;
        \end{tris}
        \node at (270:1.4cm) {$H_2$};
      \end{scope}
      \begin{scope}[xshift=1cm]
        \draw (0cm,0cm) circle (1.1cm);
        \begin{ontop}
        \avoint{w_3} (w3) at (0cm, 1.4cm) {};
        \avoint{\alv} (v) at (-1cm, 1.4cm) {};          
        \end{ontop}
        \draw (v) -- (w3);
        \apoint{} (z1) at (45:.6cm) {};
        \apoint{u} (z2) at (135:.6cm) {};
        \lpoint{r} (z3) at (225:.6cm) {};
        \rpoint{u'} (z4) at (315:.6cm) {};
        \draw[dashed] (z1) -- (z2);
        \draw[dashed] (z3) -- (z4);
        \draw[dashed] (z1) -- (z4);
        \draw[dashed] (z2) -- (z3);
        \draw (z2) -- (z4);
        \draw (w3.center) -- (z1.center);
        \begin{tris}
          \draw[triangle1] (z2.center) -- (z4.center) -- (w3.center) -- cycle;
          \draw[triangle1] (z3.center) -- (v.center) -- (w3.center) -- cycle;
        \end{tris}
        \draw (z2.center) -- (w3.center);
        \node at (270:1.4cm) {$H_3$};
      \end{scope}
      \begin{scope}[xshift=4cm]
        \draw (0cm,0cm) circle (1.1cm);
        \begin{ontop}
        \avoint{v\vphantom{w_3}} (v) at (-1cm, 1.4cm) {};
        \avoint{w'\vphantom{w_3}} (wp) at (0cm, 1.4cm) {};
        \apoint{p\vphantom{w_3}} (p) at (1cm, 1.4cm) {};          
        \end{ontop}
        \draw (wp) -- (p);
        \draw (v) -- (wp);
        \rpoint{r_2} (z1) at (45:.6cm) {};
        \lpoint{r_1} (z3) at (135:.6cm) {};
        \lpoint{r} (z4) at (225:.6cm) {};
        \draw[dashed] (z1) -- (z3);
        \draw[dashed] (z3) -- (z4);
        \begin{tris}
          \draw[triangle1] (v.center) -- (z3.center) -- (wp.center) -- cycle;
          \draw[triangle1] (wp.center) -- (z1.center) -- (p.center) -- cycle;
        \end{tris}
        \node at (270:1.4cm) {$H'$};
      \end{scope}
      \node at (-3.5cm, 0cm) {$\Rightarrow$};
      \node at (-.5cm, 0cm) {$\Rightarrow$};
      \node at (2.5cm, 0cm) {$\Rightarrow$};
    \end{tikzpicture}
    \caption{Constructing $\sey$ in Lemma~\ref{lem:9conq}.}
    \label{fig:sey9conq}
  \end{figure}
  Figure~\ref{fig:sey9conq} illustrates why the triangles in $\sey$
  are edge-disjoint. At each step, we add an edge-disjoint set of
  triangles, so it suffices to check that the triangles added in each
  step are disjoint from the earlier triangles. Since $V(H_i) \cap W =
  \emptyset$, edges incident to $w_i$ are used only in the step
  corresponding to $w_i$; similarly, edges incident to $w'$ are used
  only in the last step. By construction, we never use any edge in
  $E(H_i)$ that was previously used, so only the edges incident to $v$
  and incident to neither $w_3$ nor $w'$ are liable to be reused. The
  only such edges are $vr$, $vr_1$, and possibly $r_2p$.  Since $r$,
  $r_1$, and $r_2$ were chosen to be distinct vertices, and since $p
  \notin N[v]$ while all other vertices used in $\sey$ lie in $N[v]$,
  these edges are also distinct.

  Let $Z = N(v) - V_0$, so that $\sizeof{Z} = 5$. Define $X$ by
  \[ X =
  \begin{cases}
    E(G[Z]) \cup \{vw_1, vw_2, vw_3, vw'\}, & \text{if $N(w') - N(v) = \emptyset$;} \\
    E(G[Z]) \cup \{vw_1, vw_2, vw_3, vw', w'p, r_2p\}, & \text{if $N(w') - N(v) = \{p\}$.}
  \end{cases} \] Since $\sizeof{E(G[Z])} \leq 10$, we have $\sizeof{X}
  \leq 2\sizeof{\sey}$ in either case. By construction, $X$ contains
  every $\sey$-edge that is not incident to $V_0$. We check that $G-X$
  has no triangle containing a vertex of $V_0$. Since $E(G[Z]) \subset
  X$, any triangle in $G-X$ containing a vertex of $V_0$ must contain
  two vertices in $V_0 \cup (N(w') - N(v))$. Since $W$ is an
  independent set, the only way for a triangle to contain two such
  vertices is to contain an edge of the form $vw_i$, $vw'$, or $w'p$
  if $p$ exists. All such edges also lie in $X$; thus $V_0$ is
  reducible using $\sey$ and $X$.
\end{proof}
Finally, we prove three lemmas regarding the behavior of $7$- and
$8$-vertices.  The proofs are straightforward but require some case
analysis.
\begin{lemma}\label{lem:few-8-nbors}
  \robG. Let $uv \in E(G)$ with $d(u) \in \{7,8\}$ and $d(v) = 5$. If
  neither $\{u\}$ nor $\{v\}$ is reducible in $G$ and $u$ subsumes
  $v$, then $\{u,v\}$ is reducible in $G$.
\end{lemma}
\begin{proof}
  Let $W = N(u) \cap N(v)$ and $Z = N(u) - N[v]$.  By
  Proposition~\ref{prop:comp-matching}, $G[N(v)] \in \{K_5,
  K_5^-\}$. Hence $G[W] \in \{K_4, K_4^-\}$. Also, $\sizeof{Z}
  = d(u) - 5$, since $u$ subsumes $v$. Hence $\sizeof{Z} \in \{2,3\}$.

  If $G[W] \iso K_4^-$, then let $w_1w_2$ be the missing edge in $W$;
  otherwise, let $w_1$ and $w_2$ be distinct vertices of $W$.

  \caze{1}{$G[Z]$ contains an edge $z_1z_2$.} Let $Z^* = \{z \in Z \st d_W(z) >
  0\}$.  Observe that $\sizeof{Z^*} + \sizeof{E(G[W])} \leq 9$, with
  equality holding if and only if $\sizeof{Z^*} = 3$ and
  $\sizeof{E(G[W])} = 6$.
   
  \caze{1a}{$G[W] \iso K_4$ and $\sizeof{Z^*} = 3$.} Let $z_0$ be the
  vertex of $Z^*$ not in $\{z_1,z_2\}$. Choose $w \in N(z_0) \cap W$,
  relabeling if necessary so that $w, w_1, w_2$ are distinct, and let
  $w'$ be the unique vertex in $W - \{w, w_1,w_2\}$. \tillSX{\{u,v\}}{Figure~\ref{fig:zw-triangles}(a)}
  \begin{align*}
    \sey &= \{uw'w_1, vww_1, uvw_2, ww'w_2 \} \cup \{uz_1z_2, uz_0w\}, \\
    X &= E(G[W]) \cup \{uz : z \in Z\} \cup \{uv, z_1z_2, z_0w \}.
  \end{align*}
  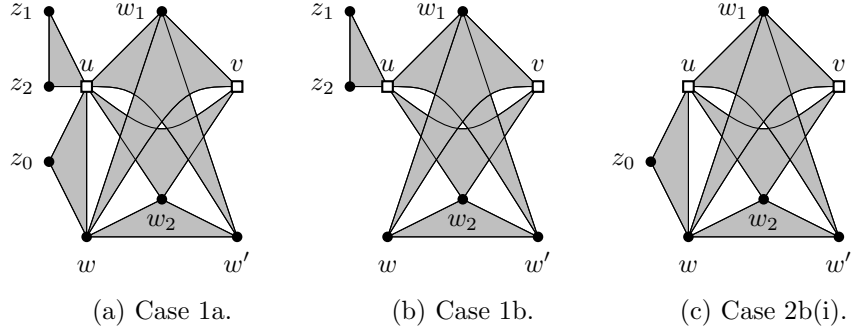
\begin{figure}
    \centering
    \begin{tikzpicture}
      \begin{scope}[xshift=-4cm]
        \begin{ontop}
          \avoint{u} (u) at (-1cm, 0cm) {};
          \avoint{v} (v) at (1cm, 0cm) {};
          \bpoint{w\vphantom{w'}} (w) at (-1cm, -2cm) {};
          \bpoint{w'} (wp) at (1cm, -2cm) {};
          \bpoint{w_2} (w2) at (0cm, -1.5cm) {};
          \lpoint{w_1} (w1) at (0cm, 1cm) {};
          \lpoint{z_0} (z) at (-1.5cm, -1cm) {};
          \lpoint{z_1} (z1) at (-1.5cm, 1cm) {};
          \lpoint{z_2} (z2) at (-1.5cm, 0cm) {};
        \end{ontop}
        \begin{tris}
          \draw[triangle1] (w.center) -- (wp.center) -- (w2.center) -- cycle;
          \draw[triangle1] (u.center) ..controls (0cm, -.75cm) .. (v.center) -- (w2.center) -- cycle;
          \draw[triangle1] (wp.center) ..controls (-.25cm, .0cm) .. (u.center) -- (w1.center) -- cycle;
          \draw[triangle1] (w.center) .. controls (.25cm, .0cm) .. (v.center) -- (w1.center) -- cycle;
          \draw[triangle1] (z1.center) -- (z2.center) -- (u.center) -- cycle;
          \draw[triangle1] (w.center) -- (z.center) -- (u.center) -- cycle;
        \end{tris}
        \draw (w.center) -- (wp.center) -- (w2.center) -- cycle;
        \draw (u.center) ..controls (0cm, -.75cm) .. (v.center) -- (w2.center) -- cycle;
        \draw (wp.center) ..controls (-.25cm, .0cm) .. (u.center) -- (w1.center) -- cycle;
        \draw (w.center) .. controls (.25cm, .0cm) .. (v.center) -- (w1.center) -- cycle;
        \draw (z1.center) -- (z2.center) -- (u.center) -- cycle;
        \draw (w.center) -- (z.center) -- (u.center) -- cycle;
        \node at (0cm, -3cm) {(a) Case~1a.};
      \end{scope}
      \begin{scope}[xshift=0cm]
        \begin{ontop}
          \avoint{u} (u) at (-1cm, 0cm) {};
          \avoint{v} (v) at (1cm, 0cm) {};
          \bpoint{w\vphantom{w'}} (w) at (-1cm, -2cm) {};
          \bpoint{w'} (wp) at (1cm, -2cm) {};
          \bpoint{w_2} (w2) at (0cm, -1.5cm) {};
          \lpoint{w_1} (w1) at (0cm, 1cm) {};
          \lpoint{z_1} (z1) at (-1.5cm, 1cm) {};
          \lpoint{z_2} (z2) at (-1.5cm, 0cm) {};
        \end{ontop}
        \begin{tris}
          \draw[triangle1] (w.center) -- (wp.center) -- (w2.center) -- cycle;
          \draw[triangle1] (u.center) ..controls (0cm, -.75cm) .. (v.center) -- (w2.center) -- cycle;
          \draw[triangle1] (wp.center) ..controls (-.25cm, .0cm) .. (u.center) -- (w1.center) -- cycle;
          \draw[triangle1] (w.center) .. controls (.25cm, .0cm) .. (v.center) -- (w1.center) -- cycle;
          \draw[triangle1] (z1.center) -- (z2.center) -- (u.center) -- cycle;
        \end{tris}
        \draw (w.center) -- (wp.center) -- (w2.center) -- cycle;
        \draw (u.center) ..controls (0cm, -.75cm) .. (v.center) -- (w2.center) -- cycle;
        \draw (wp.center) ..controls (-.25cm, .0cm) .. (u.center) -- (w1.center) -- cycle;
        \draw (w.center) .. controls (.25cm, .0cm) .. (v.center) -- (w1.center) -- cycle;
        \draw (z1.center) -- (z2.center) -- (u.center) -- cycle;
        \node at (0cm, -3cm) {(b) Case~1b.};
      \end{scope}
      \begin{scope}[xshift=4cm]
        \begin{ontop}
          \avoint{u} (u) at (-1cm, 0cm) {};
          \avoint{v} (v) at (1cm, 0cm) {};
          \bpoint{w\vphantom{w'}} (w) at (-1cm, -2cm) {};
          \bpoint{w'} (wp) at (1cm, -2cm) {};
          \bpoint{w_2} (w2) at (0cm, -1.5cm) {};
          \lpoint{w_1} (w1) at (0cm, 1cm) {};
          \lpoint{z_0} (z) at (-1.5cm, -1cm) {};
        \end{ontop}
        \begin{invis}
          \node[fill=white] at (2cm, 0cm) {};
        \end{invis}
        \begin{tris}
          \draw[triangle1] (w.center) -- (wp.center) -- (w2.center) -- cycle;
          \draw[triangle1] (u.center) ..controls (0cm, -.75cm) .. (v.center) -- (w2.center) -- cycle;
          \draw[triangle1] (wp.center) ..controls (-.25cm, .0cm) .. (u.center) -- (w1.center) -- cycle;
          \draw[triangle1] (w.center) .. controls (.25cm, .0cm) .. (v.center) -- (w1.center) -- cycle;
          \draw[triangle1] (w.center) -- (z.center) -- (u.center) -- cycle;
        \end{tris}
        \draw (w.center) -- (wp.center) -- (w2.center) -- cycle;
        \draw (u.center) ..controls (0cm, -.75cm) .. (v.center) -- (w2.center) -- cycle;
        \draw (wp.center) ..controls (-.25cm, .0cm) .. (u.center) -- (w1.center) -- cycle;
        \draw (w.center) .. controls (.25cm, .0cm) .. (v.center) -- (w1.center) -- cycle;
        \draw (w.center) -- (z.center) -- (u.center) -- cycle;
        \node at (0cm, -3cm) {(c) Case~2b(i).};
      \end{scope}
    \end{tikzpicture}
    \caption{Triangles in Lemma~\ref{lem:few-8-nbors}.}
    \label{fig:zw-triangles}
  \end{figure}%
  We check Condition~(ii) of Definition~\ref{def:reducible}. Let $T$
  be a triangle in $G-X$ containing a vertex of $\{u,v\}$. Since $E(G[W])
  \subset X$, we see that $T$ contains at most one vertex from $W$;
  hence, two vertices of $T$ must lie in $Z \cup \{u,v\}$.  If
  $v \in T$, then $T$ cannot contain any vertex of $Z$, so $\{u,v\}
  \subset T$, which is impossible since $uv \in X$. Therefore $v
  \notin T$, so $T$ contains $u$ and at least one vertex $z \in
  Z$. Since $uz \in X$, no such triangle exists.

  \caze{1b}{$\sizeof{Z^*} + \sizeof{E(G[W])} \leq 8$.}  Let $w$ and $w'$
  be the vertices of $W - \{w_1, w_2\}$. If $\sizeof{Z^*} < 2$, then
  enlarge $Z^*$ to size $2$ by adding arbitrary elements of
  $Z$. \tillSX{\{u,v\}}{Figure~\ref{fig:zw-triangles}(b)}
  \begin{align*}
    \sey &= \{uw'w_1, vww_1, uvw_2, ww'w_2 \} \cup \{uz_1z_2\}, \\
    X &= E(G[W]) \cup \{z_1z_2, uv\} \cup \{uz \st z \in Z^*\}.
  \end{align*}
  We again check Condition~(ii) of Defintion~\ref{def:reducible}.  Let
  $T$ be a triangle in $G-X$ containing a vertex of $\{u,v\}$.  As before,
  $T$ contains at most one vertex of $W$, and so $v \notin T$. Thus $u
  \in T$, so $T$ either contains two vertices of $Z$ or a vertex of
  $Z$ and a vertex of $W$. Since $\sizeof{Z^*} \geq 2$, if $T$
  contains two vertices of $Z$ then $T$ contains some vertex of $Z^*$,
  which is impossible since $uz \in X$ for all $z \in Z^*$.  On the
  other hand, if $T$ contains some $z \in Z$ and $w \in W$,
  then $d_W(z) > 0$, which implies $z \in Z^*$, again implying $uz \in X$.

  \caze{2}{$Z$ is independent.} Since $G$ is robust and $d(u) < 10$,
  $G[N(u)]$ is connected; thus no vertex of $Z$ can be isolated in
  $G[N(u)]$, so each vertex of $Z$ has a neighbor in $W$.  Let $J$ be
  the bipartite subgraph of $G$ whose partite sets are $W$ and $Z$.
  Since each vertex of $Z$ has a neighbor in $W$, we have $\alpha'(J)
  > 0$.
  \begin{figure}
    \centering
    \begin{tikzpicture}
      \begin{scope}[xshift=-4cm]
        \begin{ontop}
          \avoint{u} (u) at (0cm, 1.5cm) {};
          \lpoint{v} (v) at (-1.75cm, -1cm) {};
        \end{ontop}
        \begin{scope}[xscale=.33, xshift=-2cm, yshift=-1cm]          
          \lpoint{w} (w) at (0cm, 1cm) {};
          \apoint{t_1} (w1) at (-1cm, -1cm) {};
          \apoint{t_2} (w2) at (1cm, -1cm) {};
          \apoint{w'} (wp) at (0cm, 0cm) {};
          \draw (0cm,0cm) circle (1.75cm) {};
          \node[shape=coordinate] (c1) at (90:1.75cm) {};
          \node[shape=coordinate] (c2) at (80:1.75cm) {};
          \node[shape=coordinate] (c3) at (70:1.75cm) {};
          \node[shape=coordinate] (c4) at (60:1.75cm) {};
          \node[shape=coordinate] (e1) at (190:1.75cm) {};
          \node[shape=coordinate] (e2) at (170:1.75cm) {};
          \node[shape=coordinate] (e3) at (200:1.75cm) {};
          \node at (0cm, -2cm) {$W$};
        \end{scope}
        \begin{scope}[xscale=.33, xshift=2cm, yshift=-1cm]          
          \bpoint{z} (z) at (0cm, 1cm) {};
          \bpoint{} (z1) at (0cm, 0cm) {};
          \bpoint{} (z2) at (0cm, -1cm) {};
          \draw (0cm,0cm) circle (1.75cm) {};
          \node[shape=coordinate] (d1) at (90:1.75cm) {};
          \node[shape=coordinate] (d2) at (100:1.75cm) {};
          \node[shape=coordinate] (d3) at (110:1.75cm) {};
          \node at (0cm, -2cm) {$Z$};
        \end{scope}
        \draw (u) -- (c1);
        \draw (u) -- (c2);
        \draw (u) -- (c3);
        \draw (u) -- (c4);
        \draw (v) -- (e1);
        \draw (v) -- (e2);
        \draw (v) -- (e3);
        \draw[medge] (v) -- (wp);
        \draw[medge] (w1) -- (w2);
        \draw[medge] (w) -- (z);
        \qvert{w};
        \qvert{wp};
        \qvert{v};
        \draw (u) -- (d1);
        \draw (u) -- (d2);
        \draw (u) -- (d3); 
        \draw (u) .. controls ++(180:1cm) and ++(90:1cm) .. (v);
        \draw (w) -- (z1);
        \draw (w) -- (z2);
        \node at (0cm, -4cm) {(a) Case 2a.};
      \end{scope}
      \begin{scope}[]
        \begin{ontop}
          \avoint{u} (u) at (0cm, 1.5cm) {};
          \lpoint{v} (v) at (-1.75cm, -1cm) {};
        \end{ontop}
        \begin{scope}[xscale=.33, xshift=-2cm, yshift=-1cm]          
          \lpoint{w} (w) at (0cm, 1cm) {};
          \apoint{t_1} (w1) at (-1cm, -1cm) {};
          \apoint{t_2} (w2) at (1cm, -1cm) {};
          \apoint{w'} (wp) at (0cm, 0cm) {};
          \draw (0cm,0cm) circle (1.75cm) {};
          \node[shape=coordinate] (c1) at (90:1.75cm) {};
          \node[shape=coordinate] (c2) at (80:1.75cm) {};
          \node[shape=coordinate] (c3) at (70:1.75cm) {};
          \node[shape=coordinate] (c4) at (60:1.75cm) {};
          \node[shape=coordinate] (e1) at (190:1.75cm) {};
          \node[shape=coordinate] (e2) at (170:1.75cm) {};
          \node[shape=coordinate] (e3) at (200:1.75cm) {};
          \node[shape=coordinate] (e4) at (180:1.75cm) {};
          \node[shape=coordinate] (f) at (-10:1.75cm) {};
          \draw[dashed] (e1) -- (f);
          \node at (0cm, -2cm) {$W$};
        \end{scope}
        \begin{scope}[xscale=.33, xshift=2cm, yshift=-1cm]          
          \bpoint{} (z) at (0cm, 1cm) {};
          \bpoint{} (z1) at (0cm, 0cm) {};
          \bpoint{} (z2) at (0cm, -1cm) {};
          \draw (0cm,0cm) circle (1.75cm) {};
          \node[shape=coordinate] (d1) at (90:1.75cm) {};
          \node[shape=coordinate] (d2) at (100:1.75cm) {};
          \node[shape=coordinate] (d3) at (110:1.75cm) {};
          \node at (0cm, -2cm) {$Z$};
        \end{scope}
        \draw (u) -- (c1);
        \draw (u) -- (c2);
        \draw (u) -- (c3);
        \draw (u) -- (c4);
        \draw (v) -- (e1);
        \draw (v) -- (e2);
        \draw (v) -- (e3);
        \draw (v) -- (e4);
        \draw[medge] (w1) -- (w2);
        \draw[medge] (w) -- (z);
        \draw[medge] (wp) -- (z1);
        \qvert{w};
        \qvert{wp};
        \qvert{v};
        \draw (u) -- (d1);
        \draw (u) -- (d2);
        \draw (u) -- (d3); 
        \draw (u) .. controls ++(180:1cm) and ++(90:1cm) .. (v);
        \node at (0cm, -4cm) {(b) Case 2b(ii).};
      \end{scope}
      \begin{scope}[xshift=4cm]
        \begin{ontop}
          \avoint{u} (u) at (0cm, 1.5cm) {};
          \lpoint{v} (v) at (-1.75cm, -1cm) {};
        \end{ontop}
        \begin{scope}[xscale=.33, xshift=-2cm, yshift=-1cm]          
          \lpoint{} (w) at (0cm, 1cm) {};
          \apoint{w} (w1) at (-.8cm, -1cm) {};
          \apoint{} (w2) at (.8cm, -1cm) {};
          \apoint{} (wp) at (0cm, 0cm) {};
          \draw (0cm,0cm) circle (1.75cm) {};
          \node[shape=coordinate] (c1) at (90:1.75cm) {};
          \node[shape=coordinate] (c2) at (80:1.75cm) {};
          \node[shape=coordinate] (c3) at (70:1.75cm) {};
          \node[shape=coordinate] (c4) at (60:1.75cm) {};
          \node[shape=coordinate] (e1) at (190:1.75cm) {};
          \node[shape=coordinate] (e2) at (170:1.75cm) {};
          \node[shape=coordinate] (e3) at (200:1.75cm) {};
          \node[shape=coordinate] (e4) at (180:1.75cm) {};
          \node at (0cm, -2cm) {$W$};
        \end{scope}
        \begin{scope}[xscale=.33, xshift=2cm, yshift=-1cm]          
          \bpoint{} (z) at (0cm, 1cm) {};
          \bpoint{} (z1) at (0cm, 0cm) {};
          \bpoint{} (z2) at (0cm, -1cm) {};
          \draw (0cm,0cm) circle (1.75cm) {};
          \node[shape=coordinate] (d1) at (90:1.75cm) {};
          \node[shape=coordinate] (d2) at (100:1.75cm) {};
          \node[shape=coordinate] (d3) at (110:1.75cm) {};
          \node at (0cm, -2cm) {$Z$};
        \end{scope}
        \draw (u) -- (c1);
        \draw (u) -- (c2);
        \draw (u) -- (c3);
        \draw (u) -- (c4);
        \draw (v) -- (e1);
        \draw (v) -- (e2);
        \draw (v) -- (e4);
        \draw[medge] (w) -- (z);
        \draw[medge] (wp) -- (z1);
        \draw[medge] (w2) -- (z2);
        \draw[medge] (v) .. controls ++(270:.5cm) and ++(180:.5cm) .. (w1);
        \qvert{w};
        \qvert{wp};
        \qvert{w1};
        \qvert{w2};
        \draw (u) -- (d1);
        \draw (u) -- (d2);
        \draw (u) -- (d3);
        \draw (u) .. controls ++(180:1cm) and ++(90:1cm) .. (v);
        \node at (0cm, -4cm) {(c) Case 2c.};
      \end{scope}
    \end{tikzpicture}
    \caption{Matchings and vertex covers in Case~2.}
    \label{fig:few-8-wke}
  \end{figure}
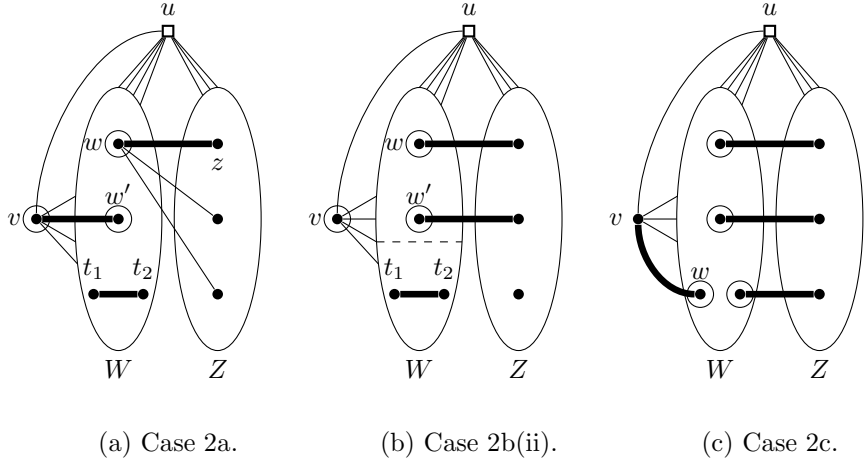
  
  \caze{2a}{$\alpha'(J) = 1$.} Since every vertex of $Z$ has a
  neighbor in $W$, $\alpha'(J) = 1$ implies that some vertex $w \in W$
  covers every edge incident to $Z$. Let $z$ be any vertex of $Z$, let
  $t_1t_2$ be an edge in $G[W]$ not containing $w$, and let $w'$ be
  the remaining vertex of $W$. Let $M = \{wz, t_1t_2, vw'\}$ and let
  $Q = \{v, w, w'\}$, as illustrated in
  Figure~\ref{fig:few-8-wke}(a). All edges incident to $Z$ are covered by
  $w$, and the only edge of $G[W]$ not covered by $\{w,w'\}$ is
  $t_1t_2$, so $Q$ is a vertex cover in $G[N(u)] - M$.  Hence
  \iswke{G[N(u)]}, contradicting the hypothesis that $\{u\}$ is not
  reducible.

  \caze{2b}{$\alpha'(J) = 2$.}  Since $\alpha'(J) = 2$,
  $\sizeof{N_W(Z)} \geq 2$.

  \caze{2b(i)}{$\sizeof{N_W(Z)} \geq 3$.}  Let $w \in N_W(Z) -
  \{w_1,w_2\}$, and pick $z_0 \in Z$ such that $wz_0 \in E(G)$. Let $w'$
  be the unique vertex in $W - \{w,w_1,w_2\}$, and let $\{q_1, q_2\}$
  be a vertex cover in
  $J$. \tillSX{\{u,v\}}{Figure~\ref{fig:zw-triangles}(c)}
  \begin{align*}
    \sey &= \{uw'w_1, vww_1, uvw_2, ww'w_2 \} \cup \{uz_0w\}, \\
    X &= E(G[W]) \cup \{uv, z_0w\} \cup \{uq_1, uq_2\}.
  \end{align*}
  The verification of Condition~(ii) is similar to Case~1b, with the
  following modifications: any bad triangle $T$ cannot contain two
  vertices of $Z$, since $G[Z]$ has no edges; and if $T$ contains a
  vertex in $Z$ and a vertex in $W$, then $T$ contains $u$ along with
  an edge in $J$, one endpoint of which lies in $\{q_1, q_2\}$. The
  other possibilities for $T$ are identical to Case~1b.

  \caze{2b(ii)}{$\sizeof{N_W(Z)} = 2$.} Let $t_1, t_2$ be the two
  vertices of $W - N_W(Z)$, and let $M$ be a maximum matching in $J$,
  as illustrated in Figure~\ref{fig:few-8-wke}(b) in the case where
  $t_1t_2 \in E(G)$. Clearly, $M$ does not cover $t_1$ or
  $t_2$. Observe that $N_W(Z) \cup \{v\}$ covers every edge in
  $G[N(u)]$, except possibly the edge $t_1t_2$ if it exists. If
  $t_1t_2 \in E(G)$, then let $M' = M \cup \{t_1t_2\}$; otherwise, let
  $M' = M \cup \{vt_1\}$. In either case, $N_W(Z) \cup \{v\}$ is a
  vertex cover of size $3$ in $G[N(u)] - M'$, so \iswke{G[N(u)]},
  contradicting the hypothesis that $\{u\}$ is not reducible.

  \caze{2c}{$\alpha'(J) = 3$.} Let $M$ be a maximum matching in $J$,
  and let $w$ be the vertex of $M$ not covered by $M$; then $M \cup
  \{vw\}$ is a matching of size $4$ in $G[N(u)]$, as shown in
  Figure~\ref{fig:few-8-wke}(c). Since $W$ is a vertex cover of size
  $4$ in $G[N(u)]$, this implies \iske{G[N(u)]}, again contradicting
  the hypothesis that $\{u\}$ is not reducible.
\end{proof}
\begin{lemma}\label{lem:6-dom-7}
  \robG. Let $uv \in E(G)$ with $d(u) = 7$ and $d(v) = 6$. If $\{v\}$
  is not reducible in $G$ and $u$ subsumes $v$, then $\{u,v\}$ is
  reducible in $G$.
\end{lemma}
\begin{proof}
  By Proposition~\ref{prop:comp-matching}, $G[N(v)] \in \{K_6,
  K_6^-\}$ (since $G[N(v)]$ has $u$ as a dominating vertex,
  $\comp{G[N(v)]}$ cannot be a perfect matching). Let $H = G[N(u) \cap N(v)]$
  and write $V(H) = \{w_1, \ldots, w_5\}$, indexed so that
  $w_1w_2$ is the possible missing edge. Let $p$ be the unique vertex
  in $N(u) - N[v]$.  \tillSX{\{u,v\}}{Figure~\ref{fig:6-dom-7}}
  \begin{align*}
    \sey &= \{uw_2w_5, uw_3w_4\} \cup \{vw_2w_3, vw_4w_5\} \cup \{uvw_1, w_1w_3w_5\}, \\
    X &= E(H) \cup \{uv, up\}.
  \end{align*}
  \begin{figure}
    \begin{tikzpicture}
      \begin{scope}[xshift=-4.5cm]
        \begin{ontop}
          \bpoint{w_1} (w1) at (-90:1cm) {};
          \bpoint{w_2} (w2) at (-45:1cm) {};
          \gpoint{below right}{w_3} (w3) at (0:1cm) {};
          \gpoint{below left}{w_4} (w4) at (-180:1cm) {};
          \bpoint{w_5} (w5) at (-135:1cm) {};
          \avoint{u} (u) at (-.5cm, .5cm) {};          
        \end{ontop}
        \begin{tris}
          \draw[triangle1] (u.center) -- (w2.center) -- (w5.center) -- cycle;
          \draw[triangle1] (u.center) -- (w3.center) -- (w4.center) -- cycle;
        \end{tris}
        \draw (u.center) -- (w2.center) -- (w5.center) -- cycle;
        \draw (u.center) -- (w3.center) -- (w4.center) -- cycle;
      \end{scope}
      \begin{scope}
        \begin{ontop}
          \bpoint{w_1} (w1) at (-90:1cm) {};
          \bpoint{w_2} (w2) at (-45:1cm) {};
          \gpoint{below right}{w_3} (w3) at (0:1cm) {};
          \gpoint{below left}{w_4} (w4) at (-180:1cm) {};
          \bpoint{w_5} (w5) at (-135:1cm) {};
          \avoint{v} (v) at (.5cm, .5cm) {};          
        \end{ontop}
        \begin{tris}
          \draw[triangle1] (v.center) -- (w2.center) -- (w3.center) -- cycle;
          \draw[triangle1] (v.center) -- (w4.center) -- (w5.center) -- cycle;
        \end{tris}
        \draw (v.center) -- (w2.center) -- (w3.center) -- cycle;
        \draw (v.center) -- (w4.center) -- (w5.center) -- cycle;
      \end{scope}
      \begin{scope}[xshift=4.5cm]
        \begin{ontop}
          \bpoint{w_1} (w1) at (-90:1cm) {};
          \bpoint{w_2} (w2) at (-45:1cm) {};
          \gpoint{below right}{w_3} (w3) at (0:1cm) {};
          \gpoint{below left}{w_4} (w4) at (-180:1cm) {};
          \bpoint{w_5} (w5) at (-135:1cm) {};
          \avoint{u} (u) at (-.5cm, .5cm) {};
          \avoint{v} (v) at (.5cm, .5cm) {};           
        \end{ontop}
        \begin{tris}
          \draw[triangle1] (u.center) -- (v.center) -- (w1.center) -- cycle;
          \draw[triangle1] (w1.center) -- (w3.center) -- (w5.center) -- cycle;
        \end{tris}
        \draw (u.center) -- (v.center) -- (w1.center) -- cycle;
        \draw (w1.center) -- (w3.center) -- (w5.center) -- cycle;
      \end{scope}
      \node at (-2.25cm, -.5cm) {$\bigcup$};
      \node at (2.25cm, -.5cm) {$\bigcup$};
    \end{tikzpicture}
    \centering    
    \caption{Triangles in Lemma~\ref{lem:6-dom-7}.}
    \label{fig:6-dom-7}
  \end{figure}
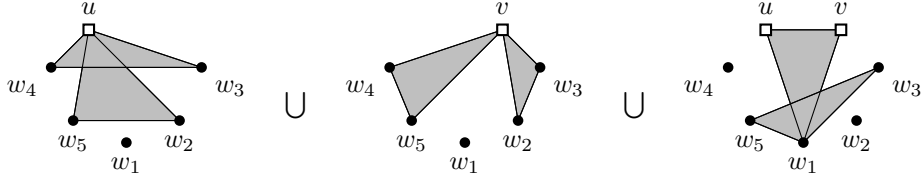%
  We check Condition~(ii) of Definition~\ref{def:reducible}. Since
  $E(H) \subset X$, any triangle of $G-X$ containing a vertex of $\{u,v\}$
  contains at most one vertex of $H$, and therefore contains two
  vertices from $\{u,v,p\}$. Since $uv, up \in X$ and $vp \notin
  E(G)$, no such triangle exists.
\end{proof}
\begin{lemma}\label{lem:6-perf-7}
  \robG. If $G$ contains an edge $uv$ such that $d(u) = 7$ and $v$ is
  a thin $6$-vertex, then $\{u,v\}$ is reducible in $G$.
\end{lemma}
\begin{proof}
  Since $\comp{G[N(v)]}$ is a matching of size $3$, we know that
  $G[N(u) \cap N(v)] \iso C_4$; let $a,b,c,d$ be the vertices of this
  cycle, listed in order.  Let $p_1,p_2$ be the two vertices of $N(u)
  - N[v]$ and let $q$ be the unique vertex in $N(v) -
  N[u]$. \illSX{\{u,v\}}{Figure~\ref{fig:6-perf-7}}
  \begin{align*}
    \sey &= \{uab, ucd, vbc, vad\}, \\
    X &= E(G[N(u) \cap N(v)]) \cup \{uv, up_1, up_2, vq \}.
  \end{align*}
  \begin{figure}
    \centering
    \begin{tikzpicture}
      \lpoint{a} (a) at (-2cm, 0cm) {};
      \apoint{b} (b) at (0cm, 2cm) {};
      \rpoint{c} (c) at (2cm, 0cm) {};
      \bpoint{d} (d) at (0cm, -2cm) {};
      \vpoint{above}{u} (u) at (-1cm, -.5cm) {};
      \vpoint{above}{v} (v) at (1cm, -.5cm) {};
      \rpoint{p_1} (p1) at (-.25cm, .5cm) {};
      \rpoint{p_2} (p2) at (-.25cm, 1cm) {};
      \lpoint{q} (q) at (.25cm, 0cm) {};
      \draw (u) -- (v);
      \begin{edges}
        \draw (u.center) -- (a.center) -- (b.center) -- cycle;
        \draw (v.center) -- (b.center) -- (c.center) -- cycle;
        \draw (u.center) .. controls (.5cm, -1.25cm) .. (c.center) -- (d.center) -- cycle;
        \draw (v.center) .. controls (-.5cm, -1.25cm) .. (a.center) -- (d.center) -- cycle;
        \draw[xedge] (a.center) -- (b.center) -- (c.center) -- (d.center) -- cycle;
        \draw[xedge] (u) -- (v);
        \draw[xedge] (u) -- (p1);
        \draw[xedge] (u) -- (p2);
        \draw[xedge] (v) -- (q);
      \end{edges}
      \begin{tris}
        \draw[triangle1] (u.center) -- (a.center) -- (b.center) -- cycle;
        \draw[triangle1] (v.center) -- (b.center) -- (c.center) -- cycle;
        \draw[triangle1] (u.center) .. controls (.5cm, -1.25cm) .. (c.center) -- (d.center) -- cycle;
        \draw[triangle1] (v.center) .. controls (-.5cm, -1.25cm) .. (a.center) -- (d.center) -- cycle;
      \end{tris}
    \end{tikzpicture}
    \caption{$\sey, X$ in Lemma~\ref{lem:6-perf-7}.}
    \label{fig:6-perf-7}
  \end{figure}%
  We quickly check Condition~(ii) of Definition~\ref{def:reducible}.
  Since $E(G[N(u) \cap N(v)]) \subset X$, any triangle of $G-X$
  containing a vertex of $\{u,v\}$ contains at most one vertex from $N(u)
  \cap N(v)$, and therefore contains two vertices from
  $\{u,v,p_1,p_2,q\}$. Let $T$ be a triangle of $G-X$ and suppose $u
  \in T$. Since $uv, up_1, up_2 \in X$ and $uq \notin E(G)$, we see
  that $T$ cannot contain two vertices of $\{u,v,p_1,p_2,q\}$, and so
  $G-X$ has no triangle containing $u$. Similar logic holds for $v$.
\end{proof}
We have now completed the proof of Lemma~\ref{lem:redlem}, giving
a list of configurations that cannot appear in a smallest counterexample
to Tuza's~Conjecture. By completing the proof of this lemma, we have
completed the proof of the main theorem, Theorem~\ref{thm:mainthm}.
\section{Acknowledgments}
The author thanks his advisor Douglas B.~West for many hours of
careful proofreading, and for helpful suggestions regarding the
design of the figures. The author also thanks the anonymous referees
for their detailed and helpful comments.

The author acknowledges support from the IC Postdoctoral Fellowship.
\bibliographystyle{elsarticle-num} \bibliography{redbib}
\end{document}